\documentclass{elsarticle}
\usepackage[english]{babel}
\usepackage{color,amsmath,amssymb,amsthm,cleveref,thm-restate,graphicx,placeins}
\usepackage[super]{nth}

\usepackage{layouts}

\biboptions{sort&compress}

\newtheorem{theorem}{Theorem}
\newtheorem{definition}[theorem]{Definition}
\newtheorem{lemma}[theorem]{Lemma}

\newcommand{\V}{\widetilde{H}^{s}\left(\Omega\right)}
\newcommand{\mat}[1]{\boldsymbol{#1}}
\newcommand{\dd}{\; d}
\newcommand{\abs}[1]{\left|#1\right|}
\newcommand{\norm}[1]{\left|\!\left|#1\right|\!\right|}
\newcommand{\restr}[2]{{% we make the whole thing an ordinary symbol
    \left.\kern-\nulldelimiterspace % automatically resize the bar with \right
      #1 % the function
      \vphantom{\big|} % pretend it's a little taller at normal size
    \right|_{#2} % this is the delimiter
  }}
\newcommand{\grad}{\nabla}
\newcommand{\dist}[1]{\operatorname{dist}\left(#1\right)}
\newcommand{\diam}[1]{\operatorname{diam}\left(#1\right)}

\begin{document}

\begin{frontmatter}
  \title{Aspects of an adaptive finite element method for the fractional Laplacian: a priori and a posteriori error estimates, efficient implementation and multigrid solver\tnoteref{dedication}\tnoteref{tn}}

  \tnotetext[dedication]{This paper is dedicated to Professor John Tinsley Oden on the occasion of his \nth{80} birthday}
  \tnotetext[tn]{This work was supported by the MURI/ARO on ``Fractional PDEs for Conservation Laws and Beyond: Theory, Numerics and Applications'' (W911NF-15-1-0562).}

  \author[brown,oakridge]{Mark Ainsworth}
  \ead{mark\_ainsworth@brown.edu}
  \author[brown]{Christian Glusa}
  \ead{christian\_glusa@brown.edu}

  \address[brown]{Division of Applied Mathematics, Brown University, 182 George St, Providence, RI 02912, USA}
  \address[oakridge]{Computer Science and Mathematics Division, Oak Ridge National Laboratory, Oak Ridge, TN 37831, USA}

\begin{abstract}
  We develop all of the components needed to construct an adaptive finite element code that can be used to approximate fractional partial differential equations, on non-trivial domains in \(d\geq 1\) dimensions.
  Our main approach consists of taking tools that have been shown to be effective for adaptive boundary element methods and, where necessary, modifying them so that they can be applied to the fractional PDE case.
  Improved a priori error estimates are derived for the case of quasi-uniform meshes which are seen to deliver sub-optimal rates of convergence owing to the presence of singularities.
  Attention is then turned to the development of an a posteriori error estimate and error indicators which are suitable for driving an adaptive refinement procedure.
  We assume that the resulting refined meshes are locally quasi-uniform and develop efficient methods for the assembly of the resulting linear algebraic systems and their solution using iterative methods, including the multigrid method.
  The storage of the dense matrices along with efficient techniques for computing the dense matrix vector products needed for the iterative solution is also considered.
  The performance and efficiency of the resulting algorithm is illustrated for a variety of examples.
\end{abstract}

  \begin{keyword}
    non-local equations \sep fractional Laplacian \sep adaptive refinement

    \MSC 65N22, 65N30, 65N38, 65N55
  \end{keyword}
\end{frontmatter}

\section{Introduction}
\label{sec:introduction}

Although the earliest works of J.T. Oden on finite element analysis date back to the 1960s, they feature the application of what was, at the time, a relatively poorly understood numerical method to problems that are regarded as challenging even by today's standards including: large deformation elasticity~\cite{OdenSato1967_FiniteStrainsDisplacementsElastic,Oden1966_AnalysisLargeDeformationsElastic}, pneumatic structures~\cite{OdenKubitza1967_NumericalAnalysisNonlinearPneumaticStructures}, thermoelasticity~\cite{Oden1969_FiniteElementAnalysisNonlinear}, fluid flow~\cite{Oden1970_FiniteElementAnalogueNavierStokesEquation}  and incompressible elasticity~\cite{OdenKey1970_NumericalAnalysisFiniteAxisymmetric}.
The analysis and application of the finite element method has come a long way in the intervening 50 years, but \emph{plus \c{c}a change, plus c'est la m\^eme chose} and one still finds the name of J.T. Oden at the cutting edge of finite element analysis with applications to tumour growth, atomistic modelling of solids and problems with multiple scales.

Oden and coworkers \cite{DemkowiczOdenEtAl1989_TowardUniversalHpAdaptive1,OdenDemkowiczEtAl1989_TowardUniversalHpAdaptive2} also promoted the use of adaptive finite element methods incorporating a posteriori error estimation and error control~\cite{AinsworthOden2011_PosterioriErrorEstimationFiniteElementAnalysis}, automatic mesh refinement and adaptivity~\cite{DemkowiczOdenEtAl1989_TowardUniversalHpAdaptive1} and efficient solvers for the result algebraic systems~\cite{ToselliWidlund2005_DomainDecompositionMethods}.
The efficiency and flexibility of adaptive finite element methods, and other current day computational techniques, opens up the possibility for the practical utilisation of ever more sophisticated mathematical models.

Recent years have witnessed a rapid increase in the use of \emph{non-local} and \emph{fractional order} models in which classical pointwise integer-order derivatives are replaced with fractional order, non-local, derivatives.
Fractional equations are used to describe phenomena in anomalous diffusion, material science, image processing, finance and electromagnetic fluids \cite{West2016_FractionalCalculusViewComplexity}.
Partial differential equations involving fractional order operators also arise naturally as the limit of discrete diffusion governed by stochastic processes \cite{MeerschaertSikorskii2012_StochasticModelsFractionalCalculus}, in the same way as standard diffusion equations arise from Brownian random walks.

Computational methods available for the numerical resolution of models involving fractional derivatives in two or more dimensions are relatively scarce, and the efficient solution of fractional equations posed on complex domains is a problem of considerable practical interest.

Our objective in the present work is to develop all of the components needed to construct an adaptive finite element code that can be used to approximate fractional partial differential equations, or, more precisely the integral fractional Laplacian, on non-trivial domains in \(d\geq 1\) dimensions.
There are a number of important differences in applying the finite element method to fractional equations as opposed to equations involving only integer-order derivatives owing to the fact that operators are non-local and involve singular integrals.
This means that (a) the computation of the entries in the stiffness matrix is non-trivial and (b) the stiffness matrix will be dense in addition to suffering from the same kind of ill-conditioning seen in the integer-order case.
Moreover, the solutions of fractional equations are inherently singular.
The solution has singularities even in the case when the domain is smooth (e.g. a disc) and the data is constant.
The non-smooth nature of the solutions mandates the use of adaptive solution algorithms.
Adaptive mesh refinement procedures require appropriate a posteriori error estimators which is also rather less straightforward than for the integer-order case.
The present work tackles each of these issues including the efficient solution of the resulting linear algebraic equations.
Fortunately, our task is alleviated considerably by realising that many of the issues described above which apply to the fractional order PDEs are similar to issues that arise in the use adaptive boundary element methods.
Our main approach consists of taking tools that have been shown to be effective for adaptive boundary element methods and, where necessary, modifying them so that they can be applied to the fractional PDE case.

In the present work, we refer to the integral fractional Laplacian (introduced in \Cref{sec:notation}) simply as the fractional Laplacian.
It is noteworthy though that this is not the only possible definition of a fractional order Laplacian on a bounded domain.
The reader is referred to \cite{NochettoOtarolaEtAl2015_PdeApproachToFractional} for an introduction to the alternative spectral definition and a finite element discretization of the corresponding extension problem in \(d+1\) spatial dimensions.
It has been shown in \cite{ServadeiValdinoci2014_SpectrumTwoDifferentFractionalOperators} that the two fractional operators are in fact not the same.

The remainder of this article is structured as follows: In \Cref{sec:notation}, we introduce the required Sobolev spaces and the weak form of the fractional Laplacian.
Regularity of solutions to the fractional Poisson problem and finite element discretization are discussed in \Cref{sec:regul-conv-theory}.
We improve upon previous a priori error estimates given in \cite{AcostaBorthagaray2015_FractionalLaplaceEquation} and determine the rate of convergence on quasi-uniform meshes.
As expected, the resulting rates are sub-optimal owing to the presence of singularities and we therefore turn our attention in \Cref{sec:adapt-mesh-refin} to the development of suitable a posteriori error estimate and error indicators.
This leads us to consider the case of non-quasi-uniform, locally refined meshes that arise from an adaptive mesh refinement procedure.
The assembly of the resulting linear algebraic systems and their solution using iterative methods, including conjugate gradient and the multigrid method, is discussed in \Cref{sec:solut-syst-involv}.
We show that, similar to the integer-order case, multigrid outperforms conjugate gradient as a solver for the fractional Poisson problem, but, under some conditions on the size of the time-step, CG can be used effectively for time-dependent problems.
In \Cref{sec:cluster-method}, we discuss the computation of the matrix entries using variable-order Gauss-type quadrature rules and the storage of the dense matrices along with efficient techniques for computing the dense matrix vector products needed for the iterative solution.
We outline the use of a cluster panelling method applied to the discretized fractional operator, and describe the efficient evaluation of the proposed residual-based error indicators.
Finally, in \Cref{sec:numerical-examples}, we give a variety of numerical examples in one and two dimensions to illustrate that assembly of the stiffness matrix, solution of the system and computation of the error indicators can be achieved in quasi-optimal complexity.
The optimal rates of convergence are shown to be attained in a series of examples involving a variety of domains and as well as right-hand sides with varying regularity.
We illustrate how symmetries or periodicity in the shape of the domain and right-hand side can be exploited to reduce the size of the linear system, and demonstrate the effect of non-locality for a domain consisting of two disconnected components.

The current work extends our previous work presented in \cite{AinsworthGlusa2017_TowardsEfficientFiniteElement} in which we considered the case of uniformly refined meshes.

\section{Definition of the fractional Laplacian and Notation}
\label{sec:notation}

The \emph{fractional Laplacian} of a function \(u: \mathbb{R}^{d}\mapsto \mathbb{R}\) is defined as
\begin{align*}
  \left(-\Delta\right)^{s} u\left(\vec{x}\right) = C(d,s) \operatorname{p.v.} \int_{\mathbb{R}^{d}} \dd \vec{y} ~ \frac{u(\vec{x})-u(\vec{y})}{\abs{\vec{x}-\vec{y}}^{d+2s}}
\end{align*}
where
\begin{align*}
  C(d,s) = \frac{2^{2s}s\Gamma\left(s+\frac{d}{2}\right)}{\pi^{d/2}\Gamma\left(1-s\right)}
\end{align*}
is a normalisation constant and \(\operatorname{p.v.}\) denotes the Cauchy principal value of the integral \cite[Chapter 5]{Mclean2000_StronglyEllipticSystemsBoundaryIntegralEquations}.
The expression for the fractional Laplacian can directly be derived from the Fourier representation \(\left(-\Delta\right)^{s}=\mathcal{F}^{-1}\abs{\vec{\xi}}^{2s}\mathcal{F}\).
In the case where \(s=1\) the operator coincides with the usual Laplacian.
If \(\Omega\subset\mathbb{R}^{d}\) is a bounded Lipschitz domain, we define the \emph{integral fractional Laplacian} \(\left(-\Delta\right)^{s}\) to be the restriction of the full-space operator to functions with compact support in \(\Omega\).
The \emph{fractional Poisson problem} is then given by
\begin{align}
  \left(-\Delta\right)^{s}u = f &\text{ in }\Omega,\label{eq:fracPoisson}\\
  u=0&\text{ in }\Omega^{c} \nonumber
\end{align}
This generalises the homogeneous integer order Poisson problem from the case \(s=1\) to the case \(s\in(0,1)\).

A related operator on \(\Omega\) is the \emph{regional fractional Laplacian} \cite{BogdanBurdzyEtAl2003_CensoredStableProcesses,ChenKim2002_GreenFunctionEstimateCensoredStableProcesses}
\begin{align*}
  \left(-\Delta\right)_{R}^{s} u\left(\vec{x}\right) = C(d,s) \operatorname{p.v.} \int_{\Omega} \dd \vec{y} ~ \frac{u(\vec{x})-u(\vec{y})}{\abs{\vec{x}-\vec{y}}^{d+2s}},
\end{align*}
which generalises the Laplacian with homogeneous Neumann boundary condition to the fractional order case.
It will be seen that the techniques developed below for the integral fractional Laplacian also apply for the regional fractional Laplacian.

In practice, we are usually interested in dynamic behaviour rather than steady state problems.
The archetype of a time-dependent problem is the \emph{fractional heat equation}
\begin{align}
  \partial_{t}u+\left(-\Delta\right)^{s}u=f & \text{ in }[0,T]\times\Omega, \label{eq:fractionalHeat}\\
  u=u_{0} & \text{ in }\left\{0\right\}\times\Omega.\nonumber
\end{align}
Equally well, we might consider the fractional heat equation with the regional fractional Laplacian instead of the integral one, if a Neumann boundary condition is more appropriate.

Define the usual fractional Sobolev space \(H^{s}\left(\mathbb{R}^{d}\right)\) via the Fourier transform.
If \(\Omega\) is a sub-domain as above, then we define the Sobolev space \(H^{s}\left(\Omega\right)\) to be \cite{Mclean2000_StronglyEllipticSystemsBoundaryIntegralEquations}
\begin{align*}
   H^{s}\left(\Omega\right)&:=\left\{u\in L^{2}\left(\Omega\right) \mid \norm{u}_{H^{s}\left(\Omega\right)} < \infty\right\},
\end{align*}
equipped with the norm
\begin{align*}
  \norm{u}_{H^{s}\left(\Omega\right)}^{2}&= \norm{u}_{L^{2}\left(\Omega\right)}^{2} + \int_{\Omega}\dd \vec{x} \int_{\Omega}\dd \vec{y} \frac{\left(u(\vec{x})-u(\vec{y})\right)^{2}}{\abs{\vec{x}-\vec{y}}^{d+2s}}.
\end{align*}
The space
\begin{align*}
  \V&:=\left\{u\in H^{s}\left(\mathbb{R}^{d}\right) \mid u=0 \text{ in } \Omega^{c}\right\}
\end{align*}
can be equipped with the energy norm
\begin{align*}
  \norm{u}_{\V} := a\left(u,u\right)^{1/2}=\sqrt{\frac{C(d,s)}{2}}\abs{u}_{H^{s}\left(\mathbb{R}^{d}\right)},
\end{align*}
which coincides with the usual fractional Sobolev norm apart from the non-standard factor \(\sqrt{C(d,s)/2}\).
For \(s>1/2\), \(\V\) coincides with the space \(H_{0}^{s}\left(\Omega\right)\) which is the closure of \(C_{0}^{\infty}\left(\Omega\right)\) with respect to the \(H^{s}\left(\Omega\right)\)-norm, whilst for \(s<1/2\), \(\V\) is identical to \(H^{s}\left(\Omega\right)\).
In the critical case \(s=1/2\), \(\V\subset H^{s}_{0}\left(\Omega\right)\), and the inclusion is strict.
(See for example \cite[Chapter 3]{Mclean2000_StronglyEllipticSystemsBoundaryIntegralEquations}.)

The fractional Poisson problem takes the variational form \cite{AinsworthGlusa2017_TowardsEfficientFiniteElement}
\begin{align}
  \text{Find } u\in\V : \quad a\left(u,v\right)=\left\langle f,v\right\rangle \quad \forall v\in\V, \label{eq:fracPoissonVariational}
\end{align}
where
\begin{align*}
  a(u,v)
  &=\frac{C(d,s)}{2} \int_{\Omega} \dd\vec{x} \int_{\Omega}\dd\vec{y}  \frac{\left(u\left(\vec{x}\right)-u\left(\vec{y}\right)\right)\left(v\left(\vec{x}\right)-v\left(\vec{y}\right)\right)}{\abs{\vec{x}-\vec{y}}^{d+2s}} \\
  &\quad+ \frac{C(d,s)}{2s} \int_{\Omega} \dd \vec{x} \int_{\partial\Omega} \dd \vec{y} \frac{u\left(\vec{x}\right) v\left(\vec{x}\right) ~ \vec{n}_{\vec{y}}\cdot\left(\vec{x}-\vec{y}\right)}{\abs{\vec{x}-\vec{y}}^{d+2s}}
\end{align*}
and where \(\vec{n}_{y}\) is the \emph{inward} normal to \(\partial\Omega\) at \(\vec{y}\).
The bilinear form associated with the regional fractional Laplacian only differs from \(a\left(\cdot,\cdot\right)\) in that the boundary term is missing.

The approximation of the integral fractional Laplacian using finite elements was considered by D'Elia and Gunzburger \cite{DEliaGunzburger2013_FractionalLaplacianOperatorBounded}.
The important work of \citet{Grubb2015_FractionalLaplaciansDomainsDevelopment} gave regularity results for the analytic solution of the fractional Poisson problem and \citet{AcostaBorthagaray2015_FractionalLaplaceEquation} obtained convergence rates for the finite element approximation supported by numerical examples computed using techniques described in \cite{AcostaBersetcheEtAl2016_ShortFeImplementation2d}.

\section{Regularity and Finite Element Discretization}
\label{sec:regul-conv-theory}

The existence of a unique solution to the fractional Poisson problem \cref{eq:fracPoissonVariational} (and its subsequent finite element approximation) follows from the Lax-Milgram Lemma.
% Results regarding the regularity of the solution were shown by \citeauthor{AcostaBorthagaray2015_FractionalLaplaceEquation} and are summarized in the following Theorem.
%
% \begin{theorem}[\cite{AcostaBorthagaray2015_FractionalLaplaceEquation}]\label{thm:regularity}
%   Let \(u\) be the solution of \eqref{eq:fracPoissonVariational}.
%   Provided that the right-hand side \(f\) has the required regularity, the solution satisfies for all \(\varepsilon>0\)
%   \begin{align*}
%     \norm{u}_{H^{s+1/2-\varepsilon}\left(\Omega\right)} & \leq
%                          \begin{cases}
%                            \frac{C\left(\Omega,s,d\right)}{\varepsilon} \norm{f}_{C^{1/2-s}\left(\Omega\right)} & \text{if } 0< s<1/2, \\
%                            \frac{C\left(\Omega,d\right)}{\varepsilon} \norm{f}_{L^{\infty}\left(\Omega\right)} & \text{if } s=1/2, \\
%                            \frac{C(\Omega,d,s,\beta)}{\sqrt{\varepsilon}(2s-1)} \norm{f}_{C^{\beta}\left(\Omega\right)} & \text{if } 1/2<s<1, \beta>0 .
%                          \end{cases}
%   \end{align*}
% \end{theorem}
The regularity of the solution was recorded in \cite{AcostaBersetcheEtAl2016_ShortFeImplementation2d} as a special case of a result by \citet{Grubb2015_FractionalLaplaciansDomainsDevelopment}:

\begin{theorem}[\cite{Grubb2015_FractionalLaplaciansDomainsDevelopment,AcostaBersetcheEtAl2016_ShortFeImplementation2d}]\label{thm:regularity}
  Let \(\partial\Omega\in C^{\infty}\), \(f\in H^{r}\left(\Omega\right)\) for \(r\geq-s\) and \(u\in\V\) be the solution of the fractional Poisson problem \eqref{eq:fracPoissonVariational}.
  Then the following regularity estimate holds:
  \begin{align*}
    u\in
    \begin{cases}
      H^{2s+r}\left(\Omega\right) & \text{if }0<s+r<1/2, \\
      H^{s+1/2-\varepsilon}\left(\Omega\right)~ \forall\varepsilon>0 & \text{if }1/2\leq s+r. \\
    \end{cases}
  \end{align*}
\end{theorem}

The result shows that increasing the regularity of the right-hand side only results in a corresponding increase of the regularity of the solution, as long as \(r<1/2-s\).
We contrast the result with the standard ``lifting property'' of the integer order Laplacian \cite[Theorem 3.10]{ErnGuermond2004_TheoryPracticeFiniteElements}, whereby  on smooth domains the solution \(u\in H^{r+2}\left(\Omega\right)\) whenever \(f\in H^{r}\left(\Omega\right)\) for all \(r\).
Examining a simple case reveals why the lifting property does not extend to fractional order.
Take \(\Omega=B(0,1)\subset\mathbb{R}^{2}\), with constant right-hand side \(f=2^{2s}\Gamma\left(1+s\right)^{2}\), so that the true solution is \(u\left(\vec{x}\right)=\left(1-\abs{\vec{x}}^{2}\right)_{+}^{s}\), where \(y_{+}=\max\left(y,0\right)\) \cite{Getoor1961_FirstPassageTimesSymmetric,DydaKuznetsovEtAl2016_FractionalLaplaceOperatorMeijerGFunction}.
Observe that \(u\) is non-smooth in the neighbourhood of the boundary, owing to the presence of a term corresponding to the fractional power of the distance to the boundary.
(See \Cref{fig:solutions}.)
\begin{figure}
  \centering
  % \printPageSize
  \includegraphics{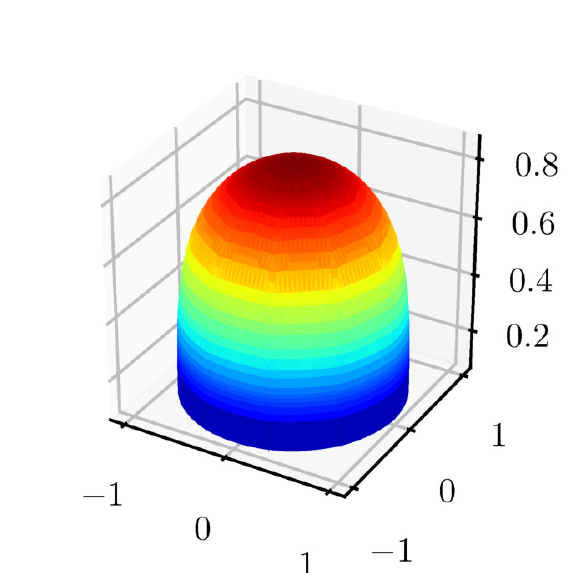}
  \includegraphics{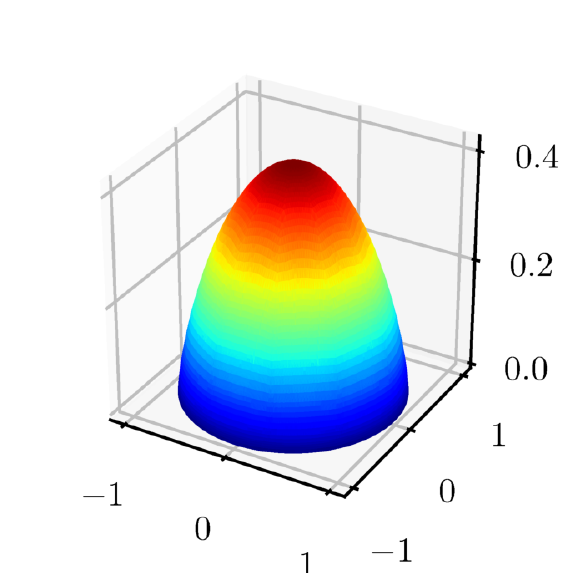}
  \caption{Solutions \(u\) corresponding to the constant right-hand side for  \(s=0.25\) and for \(s=0.75\).}
  \label{fig:solutions}
\end{figure}
Solutions of the fractional Poisson problem typically contain a term of the form \(\delta\left(\vec{x}\right)^{s}\in H^{s+1/2-\varepsilon}\), where \(\delta\left(\vec{x}\right)\) is the distance from a point \(\vec{x}\) to the boundary \(\partial\Omega\), which is not smooth when \(s\) is fractional.
Of course, for polygonal domains, the infinite lifting property does not hold in the integer order case either \cite[Theorem 3.11]{ErnGuermond2004_TheoryPracticeFiniteElements}.

Henceforth, let \(\Omega\) be a polygon, and let \(\mathcal{P}_{h}\) be a family of shape-regular and locally quasi-uniform triangulations of \(\Omega\) \cite{ErnGuermond2004_TheoryPracticeFiniteElements}, and let
\begin{align*}
  \partial\mathcal{P}_{h}=\left\{\text{edge }e\mid \exists K \in \mathcal{P}_{h}: e\subset \partial K\cap\partial\Omega \right\}
\end{align*}
be the ``trace'' of the interior mesh.
Let \(\mathcal{N}_{h}\) be the set of vertices of \(\mathcal{P}_{h}\), \(h_{K}\) be the diameter of the element \(K\in\mathcal{P}_{h}\), and \(h_{e}\) be the diameter of \(e\in\partial\mathcal{P}_{h}\).
Moreover, let
\begin{align*}
  h:=\max_{K\in\mathcal{P}_{h}}h_{K}, \\
  h_{\min}:=\min_{K\in\mathcal{P}_{h}}h_{K}, \\
  h_{\partial}:=\max_{e\in\partial\mathcal{P}_{h}}h_{e}.
\end{align*}
Let \(\phi_{i}\) be the usual piecewise linear Lagrange basis function associated with a node \(\vec{z}_{i}\in\mathcal{N}_{h}\), satisfying \(\phi_{i}\left(\vec{z}_{j}\right)=\delta_{ij}\) for \(\vec{z}_{j}\in\mathcal{N}_{h}\), and let \(X_{h}:=\operatorname{span}\left\{\phi_{i}\mid \vec{z}_{i}\in\mathcal{N}_{h}\right\}\).
The finite element subspace \(V_{h}\subset\V\) is given by \(V_{h}=X_{h}\) when \(s<1/2\) and by
\begin{align*}
  V_{h} = \left\{v_{h}\in X_{h}\mid v_{h}=0 \text{ on }\partial\Omega\right\} = \operatorname{span}\left\{\phi_{i}\mid \vec{z}_{i}\not\in \partial\Omega\right\}
\end{align*}
when \(s\geq 1/2\).
The corresponding set of degrees of freedom \(\mathcal{I}_{h}\) for \(V_{h}\) is given by \(\mathcal{I}_{h}=\mathcal{N}_{h}\) when \(s<1/2\) and otherwise consists of nodes in the interior of \(\Omega\).
In both cases we denote the cardinality of \(\mathcal{I}_{h}\) by \(n\).
The set of degrees of freedom on an element \(K\in\mathcal{P}_{h}\) is denoted by \(\mathcal{I}_{K}\).

The stiffness matrix associated with the fractional Laplacian is defined to be \(\mat{A}^{s}=\left\{a\left(\phi_{i}, \phi_{j}\right)\right\}_{i,j}\), where
\begin{align*}
  a\left(\phi_{i},\phi_{j}\right)
  &= \frac{C(d,s)}{2} \int_{\Omega} \dd \vec{x} \int_{\Omega} \dd \vec{y} \frac{\left(\phi_{i}\left(\vec{x}\right)-\phi_{i}\left(\vec{y}\right)\right)\left(\phi_{j}\left(\vec{x}\right)-\phi_{j}\left(\vec{y}\right)\right)}{\abs{\vec{x}-\vec{y}}^{d+2s}} \\
  &\quad+ \frac{C(d,s)}{2s} \int_{\Omega} \dd \vec{x} \int_{\partial\Omega} \dd \vec{y} \frac{\phi_{i}\left(\vec{x}\right) \phi_{j}\left(\vec{x}\right) ~ \vec{n}_{\vec{y}}\cdot\left(\vec{x}-\vec{y}\right)}{\abs{\vec{x}-\vec{y}}^{d+2s}}.
\end{align*}
Let \(u_{h}\in V_{h}\) be the finite element solution.
Using C\'ea's Lemma and the compactness of the embedding \(\V \hookrightarrow H^{s}\left(\Omega\right)\) (see \cite{AcostaBorthagaray2015_FractionalLaplaceEquation}), we deduce that
\begin{align}
  \norm{u-u_{h}}_{\V}&\leq C\inf_{v_{h}\in V_{h}}\norm{u-v_{h}}_{\V} \leq C\inf_{v_{h}\in V_{h}}\norm{u-v_{h}}_{H^{s}\left(\Omega\right)}, \label{eq:11}
\end{align}
which can be used to obtain an a priori error estimate by making a suitable choice of \(v_{h}\in V_{h}\).
To this end, we consider the Scott-Zhang interpolation operator with respect to the fractional \(\norm{\cdot}_{H^{s}\left(\Omega\right)}\)-norm.
For \(u\in H^{\ell}\left(\Omega\right)\), \(\ell>1/2\), the Scott-Zhang interpolation operator \cite{ScottZhang1990_FiniteElementInterpolationNonsmooth} is defined by
\begin{align*}
  \Pi_{h}u &= \sum_{i}\left(\int_{K_{i}}\psi_{i}\left(\vec{y}\right)u\left(\vec{y}\right)\dd \vec{y}\right)\phi_{i}.
\end{align*}
Here, \(K_{i}\) is either a simplex or sub-simplex containing the node \(\vec{z}_{i}\), and \(\left\{\psi_{i}\right\}_{i}\) is an \(L^{2}\left(K_{i}\right)\)-dual basis chosen such that \(\int_{K_{i}}\psi_{i}\left(\vec{y}\right)\phi_{j}\left(\vec{y}\right)\dd \vec{y}=\delta_{ij}\).
(For more details on the choice of \(K_{i}\) and properties of \(\Pi_{h}\), see \cite{ScottZhang1990_FiniteElementInterpolationNonsmooth}.)

\citet{AcostaBorthagaray2015_FractionalLaplaceEquation} showed that if \(u\in H^{\ell}\left(S_{K}\right)\), where \(S_{K}=\bigcup_{\tilde{K}\text{s.t. } \overline{K}\cap\overline{\tilde{K}}\not=\emptyset}\tilde{K}\) is the patch of an element \(K\in\mathcal{P}_{h}\), then the local approximation property
\begin{align*}
  \int_{K}\dd \vec{x} \int_{S_{K}} \dd \vec{y} \frac{\abs{\left(u-\Pi_{h}u\right)\left(\vec{x}\right) - \left(u-\Pi_{h}u\right)\left(\vec{y}\right)}^{2}}{\abs{\vec{x}-\vec{y}}^{d+2s}}
    &\leq C h_{K}^{2\ell-2s}\abs{u}_{H^{\ell}\left(S_{K}\right)}^{2}
\end{align*}
holds for \(0<s<\ell<1\) and for \(1/2<s<1\) and \(1<\ell<2\).
Below, we extend this result to include the case of \(0<s<1/2\) and \(1\leq\ell\leq2\) as well.
These extra cases are needed to take advantage of higher regularity of the solution in the interior of the domain.

\begin{lemma}\label{lem:approxProperty}
  If \(u\in H^{\ell}\left(S_{K}\right)\), for \(\ell\in (1/2,2]\) and \(0<s\leq\ell\), then
  \begin{align}
    \int_{K}\dd \vec{x} \int_{S_{K}} \dd \vec{y} \frac{\abs{\Pi_{h}u\left(\vec{x}\right) - \Pi_{h}u\left(\vec{y}\right)}^{2}}{\abs{\vec{x}-\vec{y}}^{d+2s}}
    &\leq C\left[h_{K}^{-2s}\norm{u}_{L^{2}\left(S_{K}\right)}^{2} \right.\nonumber\\
    &\left.\quad+ h_{K}^{2t-2s}\abs{u}_{H^{t}\left(S_{K}\right)}^{2}\right], \label{eq:5}\\
    \int_{K}\dd \vec{x} \int_{S_{K}} \dd \vec{y} \frac{\abs{\left(u-\Pi_{h}u\right)\left(\vec{x}\right) - \left(u-\Pi_{h}u\right)\left(\vec{y}\right)}^{2}}{\abs{\vec{x}-\vec{y}}^{d+2s}} \label{eq:6}
    &\leq C h_{K}^{2\ell-2s}\abs{u}_{H^{\ell}\left(S_{K}\right)}^{2}.
  \end{align}
\end{lemma}
\begin{proof}
  Similar to the proof in \cite{Ciarlet2013_AnalysisScott}, for \(t\in (1/2,\min\left\{1,\ell\right\}]\), we obtain
  \begin{align}
    \int_{K}\dd \vec{x} \int_{S_{K}} \dd \vec{y} \frac{\abs{\Pi_{h}u\left(\vec{x}\right) - \Pi_{h}u\left(\vec{y}\right)}^{2}}{\abs{\vec{x}-\vec{y}}^{d+2s}}
    &\leq C\sum_{i\in\mathcal{I}_{K}}\norm{\psi_{i}}_{L^{\infty}\left(K_{i}\right)}^{2}\norm{u}_{L^{1}(K_{i})}^{2}\abs{\phi_{i}}_{H^{s}\left(S_{K}\right)}^{2}. \label{eq:elemStab}
  \end{align}
  From \cite[Theorem 4.8]{AinsworthMcleanEtAl1999_ConditioningBoundaryElementEquations}, we obtain
  \begin{align}
    \abs{\phi_{i}}_{H^{s}\left(S_{K}\right)}^{2}\leq C h_{K}^{d-2s},\label{eq:elemHs}
  \end{align}
  while using \cite[Lemma 3.1]{ScottZhang1990_FiniteElementInterpolationNonsmooth} we have that
  \begin{align}
    \norm{\psi_{i}}_{L^{\infty}(K_{i})}^{2} \leq C h_{K}^{-2\dim K_{i}}.\label{eq:elemLinf}
  \end{align}
  In the case of \(K_{i}\) being a simplex of \(\mathcal{P}_{h}\), we find
  \begin{align}
    \norm{u}_{L^{1}(K_{i})}^{2}\leq h_{K}^{d}\norm{u}_{L^{2}(S_{K})}^{2}.\label{eq:elemL1K}
  \end{align}
  The case of \(K_{i}\) being a sub-simplex is more involved and covered in \cite{Ciarlet2013_AnalysisScott,ScottZhang1990_FiniteElementInterpolationNonsmooth}.
  The resulting estimate in this case is
  \begin{align}
    \norm{u}_{L^{1}(K_{i})}^{2} \leq C \left(h_{K}^{d-2}\norm{u}_{L^{2}(S_{K})}^{2} + h_{K}^{d-2+2t}\abs{u}_{H^{t}\left(S_{K}\right)}^{2}\right).\label{eq:elemL1e}
  \end{align}
  Inserting the estimates in \cref{eq:elemHs,eq:elemLinf,eq:elemL1K,eq:elemL1e} into \eqref{eq:elemStab} gives the stability estimate \eqref{eq:5}.
  % \begin{align*}
  %   \int_{K}\dd \vec{x} \int_{S_{K}} \dd \vec{y} \frac{\abs{\Pi_{h}u\left(\vec{x}\right) - \Pi_{h}u\left(\vec{y}\right)}^{2}}{\abs{\vec{x}-\vec{y}}^{d+2s}}
  %   &\leq C\left[h_{K}^{-2s}\norm{u}_{L^{2}\left(S_{K}\right)}^{2} + h_{K}^{2t-2s}\abs{u}_{H^{t}\left(S_{K}\right)}^{2}\right].
  % \end{align*}

  Now, let \(p\) be an affine function on \(S_{K}\).
  Then, by \(\Pi_{h}p=p\) and the above stability result we have
  \begin{align*}
    &\int_{K}\dd \vec{x} \int_{S_{K}} \dd \vec{y} \frac{\abs{\left(u-\Pi_{h}u\right)\left(\vec{x}\right) - \left(u-\Pi_{h}u\right)\left(\vec{y}\right)}^{2}}{\abs{\vec{x}-\vec{y}}^{d+2s}}\\
    \leq& C\left[ \int_{K}\dd \vec{x} \int_{S_{K}} \dd \vec{y} \frac{\abs{\left(u-p\right)\left(\vec{x}\right) - \left(u-p\right)\left(\vec{y}\right)}^{2}}{\abs{\vec{x}-\vec{y}}^{d+2s}}\right.\\
    &\quad \left.+ \int_{K}\dd \vec{x} \int_{S_{K}} \dd \vec{y} \frac{\abs{\left(\Pi_{h}p-\Pi_{h}u\right)\left(\vec{x}\right) - \left(\Pi_{h}p-\Pi_{h}u\right)\left(\vec{y}\right)}^{2}}{\abs{\vec{x}-\vec{y}}^{d+2s}}\right]\\
    \leq& C\left[\abs{u-p}_{H^{s}\left(S_{K}\right)}^{2} + h_{K}^{-2s}\norm{u-p}_{L^{2}\left(S_{K}\right)}^{2} + h_{K}^{2t-2s}\abs{u-p}_{H^{t}\left(S_{K}\right)}^{2}\right].
  \end{align*}
  We distinguish two cases.
  In the first case, assume that \(\ell<1\) and take \(t=\ell\).
  By the Bramble-Hilbert Lemma, there exists a constant \(p\) such that \(\norm{u-p}_{L^{2}\left(S_{K}\right)}\leq Ch_{K}^{\ell}\abs{u}_{H^{\ell}\left(S_{K}\right)}\).
  Moreover \(\abs{u-p}_{H^{\ell}\left(S_{K}\right)}=\abs{u}_{H^{\ell}\left(S_{K}\right)}\) and hence \(\norm{u-p}_{H^{\ell}\left(S_{K}\right)}\leq C\abs{u}_{H^{\ell}\left(S_{K}\right)}\).
  Let us define an operator \(T\), mapping \(u\in H^{\ell}\left(S_{K}\right)\) to \(u-p\).
  The above estimates ensure that
  \begin{align*}
    \norm{T}_{H^{\ell}\left(S_{K}\right)\rightarrow L^{2}\left(S_{K}\right)} &\leq C h_{K}^{\ell}\abs{u}_{H^{\ell}\left(S_{K}\right)},
                                                                               \intertext{and}
    \norm{T}_{H^{\ell}\left(S_{K}\right)\rightarrow H^{\ell}\left(S_{K}\right)} &\leq C\abs{u}_{H^{\ell}\left(S_{K}\right)}
  \end{align*}
  and we obtain by interpolation of operators \cite[Proposition 14.1.5]{BrennerScott2002_MathematicalTheoryFiniteElementMethods} that \(\norm{T}_{H^{\ell}\left(S_{K}\right)\rightarrow H^{s}\left(S_{K}\right)} \leq C h_{K}^{\ell-s}\abs{u}_{H^{\ell}\left(S_{K}\right)}\).
  This means that \(\abs{u-p}_{H^{s}\left(S_{K}\right)}\leq Ch_{K}^{\ell-s}\abs{u}_{H^{\ell}\left(S_{K}\right)}\).
  Therefore, we have obtain that for \(\ell<1\) it holds that
  \begin{align}
    \int_{K}\dd \vec{x} \int_{S_{K}} \dd \vec{y} \frac{\abs{\left(u-\Pi_{h}u\right)\left(\vec{x}\right) - \left(u-\Pi_{h}u\right)\left(\vec{y}\right)}^{2}}{\abs{\vec{x}-\vec{y}}^{d+2s}}
    &\leq C h_{K}^{2\ell-2s}\abs{u}_{H^{\ell}\left(S_{K}\right)}^{2}. \label{eq:localApprox}
  \end{align}

  In the second case, assume that \(\ell\in[1,2]\), and take \(t=1\).
  Again, by the Bramble-Hilbert Lemma, a polynomial \(p\) of degree one can be found such that \(\norm{u-p}_{L^{2}\left(S_{K}\right)}\leq Ch_{K}^{\ell}\abs{u}_{H^{\ell}\left(S_{K}\right)}\) and \(\abs{u-p}_{H^{1}\left(S_{K}\right)}\leq Ch_{K}^{\ell-1}\abs{u}_{H^{1}\left(S_{K}\right)}\).
  By interpolation of operators, its is obtained that \(\abs{u-p}_{H^{s}\left(S_{K}\right)}\leq Ch_{K}^{\ell-s}\abs{u}_{H^{s}\left(S_{K}\right)}\), and therefore the local approximability result \eqref{eq:localApprox} holds in this case as well.
\end{proof}

The proof of the rate of convergence of the finite element solution is now an immediate consequence of \Cref{lem:approxProperty}.
\begin{lemma}
  If \(u\in H^{t}\left(\Omega\right)\cap H_{\text{loc}}^{\ell}\left(\Omega\right)\), for \(t,\ell\in(1/2,2]\) and \(0<s\leq t\leq\ell\), i.e. if \(u\) has Sobolev regularity \(t\) and interior regularity \(\ell\), then
  \begin{align}
    \norm{u-u_{h}}_{\V}
    &\leq C \left(h^{\ell-s}\abs{u}_{H_{\text{loc}}^{\ell}\left(\Omega\right)}+h_{\partial}^{t-s}\abs{u}_{H^{t}\left(\Omega\right)}\right), \label{eq:adaptiveRate}
  \end{align}
  where \(h_{\partial}\) is the maximum size of all elements \(K\) whose patch \(S_{K}\) touches the boundary.
  In particular, if the family of triangulations \(\mathcal{P}_{h}\) is globally quasi-uniform, and \(u\in H^{t}\left(\Omega\right)\), for \(t\in (1/2,2]\) and \(0<s\leq t\), then
  \begin{align}
    \norm{u-u_{h}}_{\V}
    &\leq C h^{t-s}\abs{u}_{H^{t}\left(\Omega\right)}.\label{eq:uniformRate}
  \end{align}
\end{lemma}
\begin{proof}
  Assume that \(u\in H^{t}\left(\Omega\right)\cap H_{\text{loc}}^{\ell}\left(\Omega\right)\), for \(t,\ell\in(1/2,2]\) and \(0<s\leq t\leq\ell\).
  By a localisation result of Faermann \cite{Faermann2000_LocalizationAronszajnSlobodeckijNorm,Faermann2002_LocalizationAronszajnSlobodeckijNorm}, we can estimate
  \begin{align}
    & \norm{u-\Pi_{h}u}_{H^{s}\left(\Omega\right)}^{2} \nonumber\\
    \leq& C \sum_{K}\left[\int_{K}\dd \vec{x} \int_{S_{K}} \dd \vec{y} \frac{\abs{\left(u-\Pi_{h}u\right)\left(\vec{x}\right) - \left(u-\Pi_{h}u\right)\left(\vec{y}\right)}^{2}}{\abs{\vec{x}-\vec{y}}^{d+2s}} \right.\nonumber \\
    &\left.\qquad+ h_{K}^{-2s}\norm{u-\Pi_{h}u}_{L^{2}\left(K\right)}^{2}\right] \label{eq:localisation}.
  \end{align}
  The local approximability of the Scott-Zhang Interpolation \eqref{eq:6} then gives
  \begin{align*}
    \norm{u-\Pi_{h}u}_{H^{s}\left(\Omega\right)}^{2}
    &\leq C \left\{\sum_{\substack{K\\S_{K}\cap\partial\Omega=\emptyset}} h_{K}^{2\ell-2s}\abs{u}_{H^{\ell}\left(S_{K}\right)}^{2} + \sum_{\substack{K\\S_{K}\cap\partial\Omega\not=\emptyset}} h_{K}^{2t-2s}\abs{u}_{H^{t}\left(S_{K}\right)}^{2}\right\} \\
    &\leq C \left(h^{2\ell-2s}\abs{u}_{H_{\text{loc}}^{\ell}\left(\Omega\right)}^{2} + h_{\partial}^{2t-2s}\abs{u}_{H^{t}\left(\Omega\right)}^{2}\right),
\end{align*}
  so that we can conclude using \cref{eq:11}.
\end{proof}

Estimate \eqref{eq:uniformRate} implies that if \(u\in H^{2}\left(\Omega\right)\), then the expected rate of convergence on a globally quasi-uniform mesh is \(h^{2-s}=\mathcal{O}\left(n^{(s-2)/d}\right)\).
However, the solutions of \eqref{eq:fracPoisson} generally have limited regularity as described in \Cref{thm:regularity}, owing to the lack of regularity in the neighbourhood of the boundary.
This suggests one should use a more finely graded mesh near the boundary so that \(h_{\partial}\ll h\).
The estimate \eqref{eq:adaptiveRate} distinguishes between elements in the interior of the domain \(\Omega\) and elements touching the boundary \(\partial\Omega\).
Generally, one would hope to restore the optimal rate of convergence \(\mathcal{O}\left(n^{(s-2)/d}\right)\) observed for smooth solutions by using appropriately graded meshes.

To this end, we choose \(h_{\partial}=\mathcal{O}\left(h^{(\ell-s)/(1-2\varepsilon)}\right)\), where \(h\) denotes the maximum element size in the interior.
Applying estimate \eqref{eq:adaptiveRate} with \(t=s+1/2-\varepsilon\) we obtain
% By splitting into elements along the boundary and interior elements, we account for the expected singularity \(\delta\left(\vec{x}\right)^{s}\) along the boundary.
% This result moreover suggests that the optimal rate of convergence is \(h^{2-s}\sim n^{(s-2)/d}\).
% We will see below that in \(d\geq2\) dimensions, this is actually not true, since abrupt changes in element size between the boundary region and the interior are inhibited by the local quasi-uniformity of the mesh.
% In practice, even for well-behaved right-hand sides, only Sobolev regularity of order \(s+1/2-\varepsilon\) is to be expected, as given by \Cref{thm:regularity}.
% Therefore, using a globally quasi-uniform mesh leads to convergence with rate \(h^{1/2-\varepsilon}\).
% However, locally refined meshes can be used to obtain faster convergence, as the interior regularity of the solution \(u\) is better than the regularity up to the boundary.
% Assume that the maximum element size \(h\) is realised in the interior of the domain, and that the elements close to the boundary are of size \(h_{K} \sim h_{B}\sim h^{(\ell-s)/(1-2\varepsilon)}\).
% Then it follows that
\begin{align}
  \norm{u-u_{h}}_{\V}&\leq Ch^{\ell-s}\left(\abs{u}_{H_{\text{loc}}^{\ell}\left(\Omega\right)}+\abs{u}_{H^{1/2+s-\varepsilon}\left(\Omega\right)}\right). \label{eq:7}
\end{align}

In the one-dimensional case, \(d=1\), it is a simple matter to construct meshes for which \(h_{K}=h\) for all elements \(K\) such that \(S_{K}\cap\partial\Omega=\emptyset\), and \(h_{K}=h^{(\ell-s)/(1-2\varepsilon)}\) for the four elements whose patches touch the boundary.
Then the number of unknowns \(n=\mathcal{O}\left(h^{-1}\right)\) and therefore
\begin{align}
  \norm{u-u_{h}}_{\V}&\leq Cn^{s-\ell}\left(\abs{u}_{H_{\text{loc}}^{\ell}\left(\Omega\right)}+\abs{u}_{H^{1/2+s-\varepsilon}\left(\Omega\right)}\right). \label{eq:predictedRate1D}
\end{align}
This means that, provided \(u\in H_{\text{loc}}^{2}\left(\Omega\right)\), the optimal rate of convergence is indeed restored by the use of appropriately graded elements near the boundary to ameliorate the effect of the singularity.

The situation in the two-dimensional case is less clear-cut owing to the fact that one is constrained in the construction of the meshes if (as we do here) one prohibits the use of hanging nodes.
This raises the possibility of the optimal rate of convergence of \(\mathcal{O}\left(n^{(s-2)/d}\right)\) not being achievable in practice.
\citet{AcostaBorthagaray2015_FractionalLaplaceEquation} proposed a graded mesh for the unit disc domain for approximating the solution obtained with a constant right-hand side, in which a grading parameter \(\mu\geq1\) is selected and elements are distributed in \(M\) concentric layers of radii \(r_{i}=1-\left(1-\frac{i}{M}\right)^{\mu}\), \(i=1,\dots,M\), so that the element sizes are \(h_{i}=r_{i}-r_{i-1}\).
The dimension of the associated finite element space is \(n=\mathcal{O}\left(M^{2}\right)\) if \(\mu\leq2\), and \(n=\mathcal{O}\left(M^{\mu}\right)\) if \(\mu>2\).
% The elements of maximum size are located in the centre, and \(h\sim 1/M\).
% The elements close to the boundary have size \(1/M^{\mu}\sim h^{\mu}\).
By following the arguments in \cite{AcostaBorthagaray2015_FractionalLaplaceEquation} it is shown that
\begin{align}
  \norm{u-u_{h}}_{\V}&\leq Cn^{(\max\left\{s-\ell,-1+\varepsilon\right\})/2}\left(\abs{u}_{H_{\text{loc}}^{\ell}\left(\Omega\right)}+\abs{u}_{H^{1/2+s-\varepsilon}\left(\Omega\right)}\right). \label{eq:predictedRate2D}
\end{align}
% by choosing \(\mu=2(\ell-s)/(1-2\varepsilon)\), \eqref{eq:7} holds.
% If the order of regularity \(\ell\leq 1+s-2\varepsilon\), then \(\mu\leq2\) and hence
% \begin{align}
%   \norm{u-u_{h}}_{\V}&\leq Cn^{(s-\ell)/2}\left(\abs{u}_{H_{\text{loc}}^{\ell}\left(\Omega\right)}+\abs{u}_{H^{1/2+s-\varepsilon}\left(\Omega\right)}\right). \label{eq:predictedRate2D}
% \end{align}
% The highest rate of convergence of order \(n^{-1/2+\varepsilon}\) can be achieved if \(\ell=1+s-2\varepsilon\).
% On the other hand, if the interior regularity is of order \(\ell>1+s-2\varepsilon\), then, since \(\mu>2\), we also obtain convergence of order \(n^{-1/2+\varepsilon}\).
% This suggests that the optimal rate in \(d=2\) dimensions is \(n^{-1/2+\varepsilon}\), which can be achieved if \(u\) has interior regularity of at least order \(1+s-\varepsilon\).
This suggests that the optimal rate in \(d=2\) dimensions is \(n^{-1/2+\varepsilon}\), which can be achieved if \(u\) has interior regularity of at least order \(1+s-\varepsilon\).
We note that, contrary to the result stated in \cite{AcostaBorthagaray2015_FractionalLaplaceEquation}, this conclusion holds for all \(s\in(0,1)\), provided that the solution \(u\) enjoys sufficient interior regularity.

A priori mesh grading of the type discussed above can be beneficial but does rely on the availability of a priori knowledge of the regularity which, unfortunately, is not always the case for practical problems on complicated domains.
Therefore, we explore a posteriori error indicators for the fractional Laplacian in \Cref{sec:adapt-mesh-refin}.

\section{Adaptive Mesh Refinement}
\label{sec:adapt-mesh-refin}

The fractional Laplacian \(\left(-\Delta\right)^{s}:\V\rightarrow H^{-s}\left(\Omega\right)\) is continuous with continuous inverse.
Hence, the discretization error
\begin{align*}
  e &= u-u_{h}
\end{align*}
can be bounded from above and below in terms of the residual \(r=f-\left(-\Delta\right)^{s}u_{h}\) as
\begin{align*}
  c\norm{r}_{H^{-s}\left(\Omega\right)}\leq \norm{e}_{\V}\leq C\norm{r}_{H^{-s}\left(\Omega\right)},
\end{align*}
where \(c\) and \(C\) are positive constants depending on the norm of \(\left(-\Delta\right)^{s}\) and its inverse.
It was shown in \cite{Faermann2002_LocalizationAronszajnSlobodeckijNorm,Faermann2000_LocalizationAronszajnSlobodeckijNorm} that the Babu\v{s}ka-Rheinboldt type error indicators
\begin{align*}
  \eta_{i}:= \sup_{\substack{v\in\V\\v\phi_{i}\neq 0}} \frac{\abs{a\left(e,v\phi_{i}\right)}}{\norm{v\phi_{i}}_{\V}}
\end{align*}
are reliable and efficient for non-negative order differential operators.
Unfortunately, the estimators are not computable owing to the presence of the supremum.
However, \(\eta_{i}\) can be estimated up to a generic (unknown) factor, independent of the solution and the mesh-size, by
\begin{align}
  \hat{\eta}_{i}:=\sqrt{\sum_{K\in S_{i}} h_{K}^{2s}\norm{r}_{L^{2}\left(K\right)}^{2}},\label{eq:BRindicators}
\end{align}
where \(S_{i}\) is the patch formed by the elements which contain the node \(\vec{z}_{i}\).
Therefore, the following a posteriori estimate holds:
\begin{align*}
  \norm{e}_{\V}^{2}\leq \hat{\eta}:=\sum_{i}\hat{\eta}_{i}^{2}.
\end{align*}
% \begin{align*}
%   S_{i}:=\left\{K\in\mathcal{P}_{h}\mid \overline{K}\cap\overline{\operatorname{supp}\phi_{i}}\neq \emptyset\right\}.
% \end{align*}
The estimator \eqref{eq:BRindicators} involves only the local \(L^{2}\)-norms of the residual and can easily be evaluated using quadrature, but then one requires pointwise values for \(\left(-\Delta\right)^{s}u_{h}\).
A direct evaluation using quadrature rules is unattractive due to the presence of the singularity of the kernel \(\abs{\vec{x}-\vec{y}}^{-d-2s}\) as well as the fact that the domain of integration is unbounded.
The next result gives a method which reduces the evaluation of the strong form of the fractional Laplacian to regular integrals over bounded domains, thus circumventing the aforementioned issues.

\begin{restatable}{lemma}{strongForm}\label{lem:strongForm}
  For \(u\in V_{h}\) and \(\vec{x}\in K_{0}\) for an element \(K_{0}\in\mathcal{P}_{h}\), the fractional Laplacian \(\left(-\Delta\right)^{s}u\) can be evaluated as
  \begin{align}
    \frac{\left(-\Delta\right)^{s}u\left(\vec{x}\right)}{C\left(d,s\right)}
    &= \frac{1}{d+2s-2}\int_{\partial K_{0}} \frac{\restr{\grad u}{K_{0}}\cdot\vec{n}_{\vec{y}}}{\abs{\vec{x}-\vec{y}}^{d+2s-2}} \dd \vec{y}
      -\frac{u\left(\vec{x}\right)}{2s} \int_{\partial K_{0}} \frac{\vec{n}_{\vec{y}}\cdot\left(\vec{x}-\vec{y}\right)}{\abs{\vec{x}-\vec{y}}^{d+2s}} \dd \vec{y} \nonumber\\
    &\quad + \sum_{K\neq K_{0}} \left\{\frac{1}{2s(d+2s-2)} \int_{\partial K}\restr{\grad u}{K} \cdot \vec{n}_{\vec{y}} \frac{1}{\abs{\vec{x}-\vec{y}}^{d+2s-2}}\dd \vec{y} \right. \nonumber\\
      &\qquad\left.- \frac{1}{2s}\int_{\partial K} u\left(\vec{y}\right) \frac{\vec{n}_{\vec{y}}\cdot\left(\vec{x}-\vec{y}\right)}{\abs{\vec{x}-\vec{y}}^{d+2s}}\dd \vec{y}\right\} \label{eq:8}
  \end{align}
  if \(s\neq 1-d/2\), and as
  \begin{align}
    \frac{\left(-\Delta\right)^{s}u\left(\vec{x}\right)}{C\left(d,s\right)}
    &= \int_{\partial K_{0}} \restr{\grad u}{K_{0}}\cdot\vec{n}_{\vec{y}} \log \frac{1}{\abs{\vec{x-\vec{y}}}} \dd \vec{y}
      -\frac{u\left(\vec{x}\right)}{2s} \int_{\partial K_{0}} \frac{\vec{n}_{\vec{y}}\cdot\left(\vec{x}-\vec{y}\right)}{\abs{\vec{x}-\vec{y}}^{d+2s}} \dd \vec{y} \nonumber\\
    &\quad + \sum_{K\neq K_{0}} \left\{\frac{1}{2s} \int_{\partial K}\restr{\grad u}{K} \cdot \vec{n}_{\vec{y}} \log\frac{1}{\abs{\vec{x}-\vec{y}}}\dd \vec{y} \right.\nonumber \\
   &\qquad\left.   -\frac{1}{2s}\int_{\partial K}  u\left(\vec{y}\right) \frac{\vec{n}_{\vec{y}}\cdot\left(\vec{x}-\vec{y}\right)}{\abs{\vec{x}-\vec{y}}^{2}}\dd \vec{y}  \right\}.\label{eq:9}
  \end{align}
  when \(s=1-d/2\).
\end{restatable}
The derivation of the above expressions is given in \Cref{sec:proof}.
The integrals over the element surface in \cref{eq:8,eq:9} can be evaluated using standard quadrature rules since the integrals in \cref{eq:8,eq:9} are regular.
Although we are primarily concerned here with the case of piecewise linear basis functions, the above expression can easily be generalised to higher order finite elements.

\section{Computation of Matrix Entries and Solution of Systems Arising from the Discretization of the Fractional Laplacian}
\label{sec:solut-syst-involv}

The fractional Poisson equation \eqref{eq:fracPoisson} leads to a dense linear algebraic system
\begin{align}
  \mat{A}^{s}\vec{u}=\vec{b}, \label{eq:systemPoisson}
\end{align}
in which the entries in the matrix \(\mat{A}^{s}=\left\{a\left(\phi_{i},\phi_{j}\right)\right\}_{ij}\) involve singular integrals over \(\Omega\times\Omega\).
In order to compute these entries, we decompose the integrals over \(\Omega\times\Omega\) into contributions between pairs of elements \(K,\tilde{K}\in\mathcal{P}_{h}\) and between pairs consisting of elements \(K\in\mathcal{P}_{h}\) and external edges \(e\in\partial\mathcal{P}_{h}\) as follows:
\begin{align*}
  a(\phi_{i},\phi_{i}) &= \sum_{K}\sum_{\tilde{K}} a^{K\times\tilde{K}}(\phi_{i},\phi_{j}) + \sum_{K}\sum_{e} a^{K\times e}(\phi_{i},\phi_{j}),
\end{align*}
The individual contributions \(a^{K\times\tilde{K}}\) and \(a^{K\times e}\) are given by:
\begin{align}
  a^{K\times\tilde{K}}(\phi_{i},\phi_{j})
  & = \frac{C(d,s)}{2} \int_{K} \dd \vec{x} \int_{\tilde{K}} \dd \vec{y} \frac{\left(\phi_{i}(\vec{x})-\phi_{i}(\vec{y})\right)\left(\phi_{j}(\vec{x})-\phi_{j}(\vec{y})\right)}{\abs{\vec{x}-\vec{y}}^{d+2s}}, \label{eq:elementPair}\\
    a^{K\times e}(\phi_{i},\phi_{j})
  &= \frac{C(d,s)}{2s} \int_{K} \dd \vec{x} \int_{e} \dd \vec{y} \frac{\phi_{i}\left(\vec{x}\right) \phi_{j}\left(\vec{x}\right) ~ \vec{n}_{e}\cdot\left(\vec{x}-\vec{y}\right)}{\abs{\vec{x}-\vec{y}}^{d+2s}}.\label{eq:elementPairBoundary}
\end{align}
Contributions from non-disjoint pairs of elements are not amenable to numerical quadrature.
Fortunately, these can be treated by adapting techniques used in the boundary element method literature to address similar issues arising from singular kernels \cite{SauterSchwab2010_BoundaryElementMethods}.
However, the fractional Laplacian does pose new difficulties beyond those addressed by the BEM literature, but which can be treated as described in \cite{AinsworthGlusa2017_TowardsEfficientFiniteElement,AcostaBersetcheEtAl2016_ShortFeImplementation2d}.
In particular, in \cite{AinsworthGlusa2017_TowardsEfficientFiniteElement} we develop non-uniform order Gauss-type quadrature rules applicable to the case of globally quasi-uniform meshes.
The extension of these techniques to the non-quasi-uniform case needed here is described in \Cref{sec:determ-quadr-orders}.
Alternatively, an approach proposed by \citeauthor{ChernovPetersdorffEtAl2011_ExponentialConvergenceHpQuadrature} could be taken.
The algorithm put forward in \cite{ChernovSchwab2012_ExponentialConvergenceGauss,ChernovPetersdorffEtAl2015_QuadratureAlgorithmsHighDimensional,ChernovPetersdorffEtAl2011_ExponentialConvergenceHpQuadrature} allows to transform quadrature rules given on the unit hypercube \([0,1]^{2d}\) to any pair of elements \(K\times\tilde{K}\).
Here, the singularity is taken into account through the choice of the weight of the quadrature rules.

The solution of the dense matrix equation \eqref{eq:systemPoisson} also poses problems.
For instance, using matrix factorisation to solve the resulting system has complexity \(\mathcal{O}\left(n^{3}\right)\), where \(n\) is the number of degrees of freedom.

Alternatively, if conjugate gradient iteration is used, the number of iterations necessary to converge to a given error tolerance scales as the square root of the condition number of the matrix.
The condition numbers of the fractional Laplacian grows with the number of unknowns \(n\), as shown by \Cref{thm:condFracLapl}.
\begin{theorem}[\cite{AinsworthMcleanEtAl1999_ConditioningBoundaryElementEquations}]\label{thm:condFracLapl}
  For \(s<d/2\), and a family of shape regular triangulations \(\mathcal{P}_{h}\) with minimal and maximal element size \(h_{\min}\) and \(h\), the spectrum of the stiffness matrix satisfies
  \begin{align*}
    cn^{-2s/d}h_{\min}^{d-2s}\mat{I} \leq \mat{A}^{s} \leq Ch^{d-2s}\mat{I}, \\
    cn^{-2s/d}\mat{I} \leq \left(\mat{D}^{s}\right)^{-1/2}\mat{A}^{s} \left(\mat{D}^{s}\right)^{-1/2} \leq C\mat{I},
  \end{align*}
  where \(\mat{D}^{s}\) is the diagonal part of \(\mat{A}^{s}\).
  The condition number of the stiffness matrix satisfies
  \begin{align}
    \kappa\left(\mat{A}^{s}\right) &= C \left(\frac{h}{h_{\min}}\right)^{d-2s}n^{2s/d}, \label{eq:cond}\\
    \kappa\left(\left(\mat{D}^{s}\right)^{-1}\mat{A}^{s}\right) &= Cn^{2s/d}.\label{eq:condPreconditioned}
  \end{align}
\end{theorem}

The exponent of the growth of the condition number depends on the fractional order \(s\).
For small \(s\), the growth is slower.
For large \(s\), the growth of the condition number approaches that of the usual integer order Laplacian.
Moreover, \eqref{eq:cond} shows that the ill-conditioning becomes increasingly severe on non-uniform meshes.
Fortunately, estimate \eqref{eq:condPreconditioned} shows that the simple expedient of using diagonal scaling as a preconditioner removes the effects of the non-uniformity of the triangulation.
The conjugate gradient method applied to \cref{eq:systemPoisson} is therefore expected to converge in \(\sqrt{\kappa\left(\left(\mat{D}^{s}\right)^{-1}\mat{A}^{s}\right)}=\mathcal{O}\left(n^{s/d}\right)\) iterations.
Alternatively, a geometric multigrid solver can be used.
In the integer order case, multigrid iterations are used to good effect for solving systems involving both the mass matrix and the stiffness matrix which arises from the discretization of the regular Laplacian.
It is known that multigrid converges with a rate that is independent of the problem size for pseudo-differential operators of positive order in general, which in particular includes the fractional Laplacian \cite{SauterSchwab2010_BoundaryElementMethods,Hackbusch1985_MultiGridMethodsApplications,Hackbusch1994_IterativeSolutionLargeSparseSystemsEquations,AinsworthMclean2003_MultilevelDiagonalScalingPreconditioners}.
Although a single multigrid iteration is more expensive than a single iteration of conjugate gradient, a multigrid solver is more efficient, especially for larger fractional order \(s\) that are troublesome for the conjugate gradient method.

As mentioned earlier, the transient problem \eqref{eq:fractionalHeat} is often of more interest than the steady state problem \eqref{eq:fracPoisson}.
We first observe that using an explicit schemes for the time discretization will lead to a CFL condition on the time-step size \(\Delta t\) of the form \(\Delta t \leq C h^{2s}\).
The discretization using an implicit integration scheme in the time variable leads to systems of the form
\begin{align}
    \left(\mat{M}+\Delta t\mat{A}^{s}\right)\vec{u}=\vec{b},\label{eq:systemHeat}
\end{align}
where \(\mat{M}\) is the mass matrix.
In typical examples, the time-step will be chosen so that the error from the approximation in time matches the spatial order, and hence \(\Delta t\) depends on the spatial discretization.
The following theorem shows that if the time-step size \(\Delta t=\mathcal{O}\left(h^{2s}\right)\), then the condition number of the system is bounded by a constant and a simple conjugate gradient iteration can be used as an effective solver.
\begin{lemma}\label{lem:condFracLaplTime}
  For a shape regular and globally quasi-uniform family of triangulations \(\mathcal{P}_{h}\) and time-step \(\Delta t \leq 1\),
  \begin{align*}
    \kappa\left(\mat{M}+\Delta t\mat{A}^{s}\right)
    &\leq C \left(1+\frac{\Delta t}{h^{2s}}\right).
  \end{align*}
  More generally, for a family of triangulations that is only locally quasi-uniform and \(\Delta t\leq h_{\min}^{2s}n^{2s/d}\), it holds that
  \begin{align}
    \kappa\left(\mat{M}+\Delta t\mat{A}^{s}\right)
    &\leq C\left(\frac{h}{h_{\min}}\right)^{d} \left(1+\frac{\Delta t}{h^{2s}}\right). \label{eq:condHeat}
  \end{align}
  If \(\mat{D}^{0}\) is taken to be the diagonal part of the mass matrix and \(\Delta t \leq h^{2s}n^{2s/d}\), then
  \begin{align}
    \kappa\left(\left(\mat{D}^{0}\right)^{-1}\left(\mat{M}+\Delta t \mat{A}^{s}\right)\right)
    &\leq C\left(1+\frac{\Delta t}{h_{\min}^{2s}}\right) \label{eq:condPreconditionedHeat}
  \end{align}
\end{lemma}
\begin{proof}
  Since \(ch_{\min}^{d}\mat{I}\leq \mat{M} \leq Ch^{d}\mat{I}\), this also permits us to deduce that
  \begin{align*}
    c\left(h_{\min}^{d} +  \Delta t~ n^{-2s/d}h_{\min}^{d-2s}\right)\mat{I} \leq \mat{M}+\Delta t \mat{A}^{s} \leq C\left(h^{d} + \Delta t~ h^{d-2s}\right)\mat{I}
  \end{align*}
  and so
  \begin{align*}
    \kappa\left(\mat{M}+\Delta t\mat{A}^{s}\right)
    &\leq C \frac{h^{d}+\Delta t~ h^{d-2s}}{h_{\min}^{d}+\Delta t~ n^{-2s/d}h_{\min}^{d-2s}} \\
    &= C \left(\frac{h}{h_{\min}}\right)^{d} \frac{1+\Delta t~ h^{-2s}}{1+\Delta t~ n^{-2s/d}h_{\min}^{-2s}} \\
    &\leq  C \left(\frac{h}{h_{\min}}\right)^{d} \left(1+\frac{\Delta t}{h^{2s}}\right).
  \end{align*}
  For globally quasi-uniform meshes, \(h/h_{\min}=\mathcal{O}\left(1\right)\) and \(h_{\min}^{2s}n^{2s/d}=\mathcal{O}\left(1\right)\).

  Since \(\norm{\phi_{i}}_{L^{2}}^{2}=\mathcal{O}\left(h_{i}^{d}\right)\) and \(\norm{\phi_{i}}_{\widetilde{H}^{s}}^{2}=\mathcal{O}\left(h_{i}^{d-2s}\right)\) \cite{AinsworthMcleanEtAl1999_ConditioningBoundaryElementEquations} and \(c\mat{I}\leq\left(\mat{D}^{0}\right)^{-1/2}\mat{M}\left(\mat{D}^{0}\right)^{-1/2}\leq C\mat{I}\), we find
  \begin{align*}
    c\left(1 +  \Delta t~ n^{-2s/d}h^{-2s}\right)\mat{I} \leq \left(\mat{D}^{0}\right)^{-1/2}\left(\mat{M}+\Delta t \mat{A}^{s}\right)\left(\mat{D}^{0}\right)^{-1/2} \leq C\left(1 + \Delta t~ h_{\min}^{-2s}\right)\mat{I},
  \end{align*}
  and hence
  \begin{align*}
    \kappa\left(\left(\mat{D}^{0}\right)^{-1}\left(\mat{M}+\Delta t \mat{A}^{s}\right)\right)
    &\leq C\frac{1+\Delta t ~h_{\min}^{-2s}}{1+\Delta t~ n^{-2s/d}h^{-2s}}
      \leq C\left(1+\frac{\Delta t}{h_{\min}^{2s}}\right).
  \end{align*}
\end{proof}
\Cref{lem:condFracLaplTime} shows that if the time step \(\Delta t\) is chosen to be \(\Delta t =\mathcal{O}\left(h_{\min}^{2s}\right)\), then using the conjugate gradient method to solve the diagonally preconditioned system \eqref{eq:systemHeat} gives a uniform bound on the number of iterations.
However, if larger time steps are chosen, e.g. independent of \(h_{min}\), then \eqref{eq:condPreconditionedHeat} shows that the simple diagonal scaling becomes inefficient resulting in a number of iterations growing as \(\mathcal{O}\left(\Delta t/h_{\min}^{2s}\right)\).
Nevertheless, by applying a multigrid solver (as in the steady state case), one can restore a uniform bound on the number of iterations \cite{SauterSchwab2010_BoundaryElementMethods,Hackbusch1985_MultiGridMethodsApplications,Hackbusch1994_IterativeSolutionLargeSparseSystemsEquations,AinsworthMclean2003_MultilevelDiagonalScalingPreconditioners}.

% For a general time-step \(\Delta t\geq h^{2s}\), the number of iterations will grow as \(\sqrt{\frac{\Delta t}{h^{2s}}} \sim n^{s/d}\sqrt{\Delta t} \).

% Similar to the steady state problem, a multigrid solver can be used for \eqref{eq:systemHeat} as well.
% If \(\Delta t\) is large compared to \(h^{2s}\), multigrid outperforms conjugate gradient for the systems \eqref{eq:systemHeat}, whereas for values of \(\Delta t\) that are small compared to \(h^{2s}\) the conjugate gradient method can be more efficient.

In this section we have concerned ourselves with the effect the non-uniformity of the mesh and the fractional order have on the rate of convergence of iterative solvers.
This, of course, ignores the cost of carrying out the iteration.
The complexity of both multigrid and conjugate gradient iterations depends on how efficiently the matrix-vector product \(\mat{A}^{s}\vec{x}\) can be computed.
(The mass matrix in \cref{eq:systemHeat} has \(\mathcal{O}\left(n\right)\) entries, so its matrix-vector product scales linearly in the number of unknowns.)
Since all the basis functions \(\phi_{i}\) interact with each other, the matrix \(\mat{A}^{s}\) is dense and its matrix-vector product has complexity \(n^{2}\).
In the following section, we discuss a sparse approximation that will preserve the order of the approximation error of the fractional Laplacian, but display significantly better scaling in memory and operations for both assembly and matrix-vector product.

\section{Sparse Approximation of the Stiffness Matrix and of the Strong Form of the Fractional Laplacian}
\label{sec:cluster-method}

The presence of a factor \(\abs{\vec{x}-\vec{y}}^{-d-2s}\) in the integrand in the expression for the entries of the stiffness matrix means that the contribution of pairs of elements that are well separated is significantly smaller than the contribution arising from pairs of elements that are close to one another.
The evaluation of the strong form of the fractional Laplacian as given in \cref{eq:8,eq:9} involves integration over all edges of the mesh weighted by the same kernel.
This suggests the use of the \emph{cluster method} from the boundary element literature, whereby such far field contributions are replaced by less expensive low-rank blocks rather than computing and storing the all individual entries from the original matrix.
Conversely, the near-field contributions are more significant but involve only local couplings and hence the cost of storing the individual entries is a practical proposition.
A full discussion of the clustering method is beyond the scope of the present work but can be found in \cite[Chapter 7]{SauterSchwab2010_BoundaryElementMethods}.
Here, we confine ourselves to stating only the necessary definitions and steps needed to describe our approach.

\begin{definition}[\cite{SauterSchwab2010_BoundaryElementMethods}]
  A \emph{cluster} is a union of one or more indices from the set of degrees of freedom \(\mathcal{I}\).
  The nodes of a hierarchical cluster tree \(\mathcal{T}\) are clusters.
  The set of all nodes is denoted by \(T\) and satisfies
  \begin{enumerate}
  \item \(\mathcal{I}\) is a node of \(\mathcal{T}\).
  \item The set of leaves \(\text{Leaves}(\mathcal{T})\subset T\) corresponds to the degrees of freedom \(i\in\mathcal{I}\) and is given by
    \begin{align*}
      \text{Leaves}(\mathcal{T}) := \left\{\left\{i\right\} : i\in\mathcal{I}\right\}.
    \end{align*}
  \item For every \(\sigma\in T\setminus \text{Leaves}\left(\mathcal{T}\right)\) there exists a minimal set \(\Sigma\left(\sigma\right)\) of nodes in \(T\setminus\left\{\sigma\right\}\) that satisfies
    \begin{align*}
      \sigma = \bigcup_{\tau\in\Sigma\left(\sigma\right)}\tau.
    \end{align*}
    The set \(\Sigma\left(\sigma\right)\) is called the sons of \(\sigma\).
    The edges of the cluster tree \(\mathcal{T}\) are the pairs of nodes \(\left(\sigma,\tau\right)\in T\times T\) such that \(\tau\in \Sigma\left(\sigma\right)\).
  \end{enumerate}
\end{definition}
An example of a cluster tree for a one-dimensional problem is given in \Cref{fig:clusterTree}.

\begin{figure}
  \centering
  % created using dodo.py
  \includegraphics{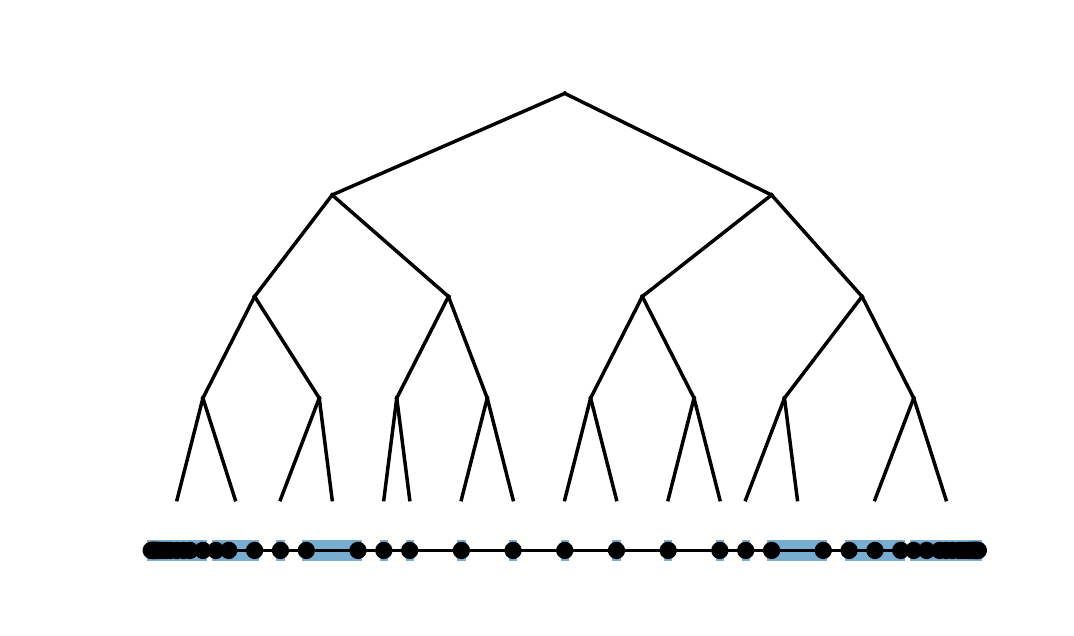}
  \caption{
    Cluster tree for a one-dimensional problem.
    The mesh with its nodal degrees of freedom is plotted at the bottom.
    The leaf clusters are coloured in blue.
  }
  \label{fig:clusterTree}
\end{figure}

\begin{definition}[\cite{SauterSchwab2010_BoundaryElementMethods}]
  The \emph{cluster box} \(Q_{\sigma}\) of a cluster \(\sigma\in T\) is the minimal hyper-cube which contains \(\bigcup_{i\in\sigma}\operatorname{supp} \phi_{i}\).
  The \emph{diameter} of a cluster is the diameter of its cluster box \(\diam \sigma:= \sup_{\vec{x},\vec{y}\in Q_{\sigma}} \abs{\vec{x}-\vec{y}}\).
  The \emph{distance} of two clusters \(\sigma\) and \(\tau\) is \(\dist{\sigma,\tau}:=\inf_{\vec{x}\in Q_{\sigma}, \vec{y}\in Q_{\tau}}\abs{\vec{x}-\vec{y}}\).
  The subspace \(V_{\sigma}\) of \(V_{h}\) is defined as \(V_{\sigma}:=\operatorname{span}\left\{\phi_{i}\mid i\in\sigma\right\}\).
\end{definition}

For given \(\eta>0\), a pair of clusters \(\left(\sigma,\tau\right)\) is called \emph{admissible}, if
\begin{align*}
  \eta\dist{\sigma,\tau}& \geq \max\left\{\diam{\sigma}, \diam{\tau}\right\}.
\end{align*}
The admissible cluster pairs can be determined recursively.
Cluster pairs that are not admissible and have no admissible sons are part of the near field and are assembled into a sparse matrix.
The admissible cluster pairs for a one dimensional problem are shown in \Cref{fig:clusterPairs}.

\begin{figure}
  \centering
  % created using dodo.py
  \includegraphics{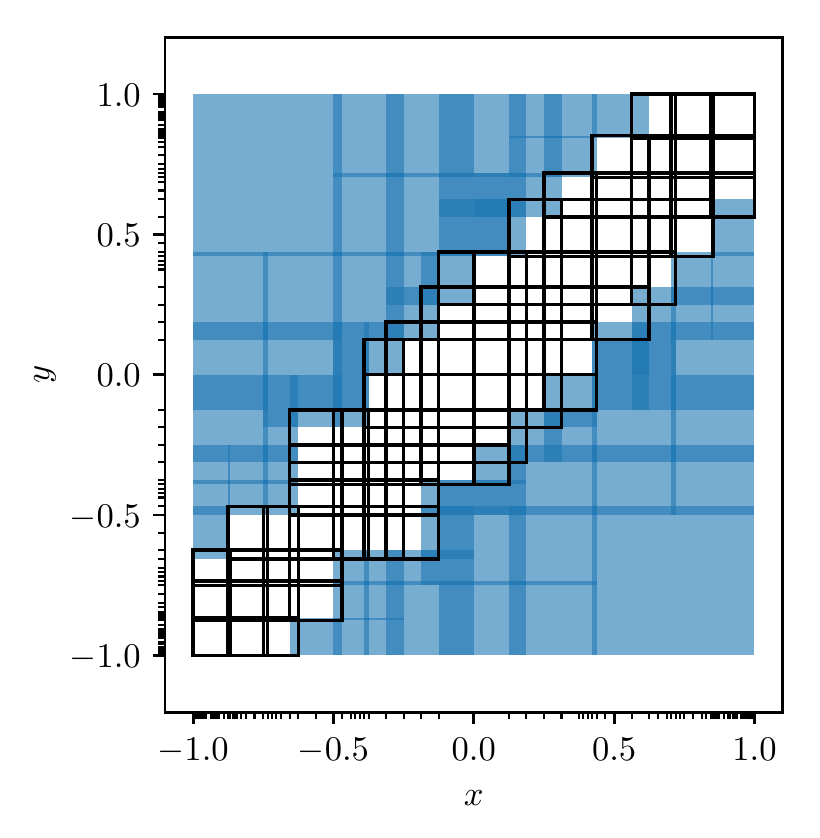}
  \caption{
    Cluster pairs for a one-dimensional problem.
    The cluster boxes of the admissible cluster pairs are coloured in light blue, and their overlap in darker blue.
    The diagonal cluster pairs are not admissible and are not approximated, but assembled in full.
    The nodes of the nodal degrees of freedom are plotted on \(x\)- and \(y\)-axis.
  }
  \label{fig:clusterPairs}
\end{figure}

For admissible pairs of clusters \(\sigma\) and \(\tau\) and any functions \(\phi,\psi\) with \(\operatorname{supp}\phi\subset Q_{\sigma}\) and \(\operatorname{supp}\psi\subset Q_{\tau}\), the bilinear form \(a\) evaluates to
\begin{align*}
  a\left(\phi,\psi\right) =-C\left(d,s\right)\int_{\Omega} \int_{\Omega} k\left(\vec{x},\vec{y}\right)\phi\left(\vec{x}\right) \psi\left(\vec{y}\right),
\end{align*}
since \(Q_{\sigma}\cap Q_{\tau}=\emptyset\).
Here, \(k\left(\vec{x},\vec{y}\right)=\abs{\vec{x}-\vec{y}}^{-(d+2s)}\) is the interaction kernel, which can be approximated on \(Q_{\sigma}\times Q_{\tau}\) using Chebyshev interpolation of order \(m\) in every spatial dimension by
\begin{align*}
  k_{m}\left(\vec{x},\vec{y}\right)&= \sum_{\alpha,\beta=1}^{m^{d}} k\left(\vec{\xi}_{\alpha}^{\sigma},\vec{\xi}_{\beta}^{\tau}\right) L_{\alpha}^{\sigma}\left(\vec{x}\right) L_{\beta}^{\tau}\left(\vec{y}\right).
\end{align*}
\(L_{\alpha}^{\sigma}\) are the tensor Chebyshev polynomials of order \(m\) on the cluster box \(Q_{\sigma}\), and \(\vec{\xi}_{\alpha}^{\sigma}\) are the tensor Chebyshev points on \(Q_{\sigma}\) with \(L_{\alpha}^{\sigma}\left(\vec{\xi}_{\beta}^{\sigma}\right)=\delta_{\alpha\beta}\).
This leads to the following approximation:
\begin{align*}
  a\left(\phi,\psi\right)
  &\approx  -C\left(d,s\right) \sum_{\alpha,\beta=1}^{m^{d}} k\left(\vec{\xi}_{\alpha}^{\sigma},\vec{\xi}_{\beta}^{\tau}\right) \int_{\operatorname{supp} \phi} \phi\left(\vec{x}\right)L_{\alpha}^{\sigma}\left(\vec{x}\right) \dd \vec{x} \int_{\operatorname{supp} \psi} \psi\left(\vec{y}\right) L_{\beta}^{\tau}\left(\vec{y}\right) \dd\vec{y}.
\end{align*}
Expressions of the type \(\int_{\operatorname{supp} \phi} \phi\left(\vec{x}\right)L_{\alpha}^{\sigma}\left(\vec{x}\right) \dd \vec{x}\) can be computed recursively starting from the finest level of the cluster tree, since for \(\tau\in\Sigma\left(\sigma\right)\) and \(\vec{x}\in Q_{\tau}\)
\begin{align*}
  L_{\alpha}^{\sigma}\left(\vec{x}\right)&=\sum_{\beta}L_{\alpha}^{\sigma}\left(\vec{\xi}_{\beta}^{\tau}\right) L_{\beta}^{\tau}\left(\vec{x}\right).
\end{align*}
This means that for all leaves \(\sigma=\left\{i\right\}\), and all \(1\leq\alpha\leq m^{d}\), the \emph{basis far-field coefficients}
\begin{align*}
  \int_{\operatorname{supp} \phi\cap Q_{\sigma}} \phi\left(\vec{x}\right)L_{\alpha}^{\sigma}\left(\vec{x}\right) \dd \vec{x}
\end{align*}
need to be evaluated (e.g. by \(m+1\)-th order Gaussian quadrature).
Moreover, the \emph{shift coefficients}
\begin{align*}
  L_{\alpha}^{\sigma}\left(\vec{\xi}_{\beta}^{\tau}\right)
\end{align*}
for \(\tau\in\Sigma\left(\sigma\right)\) must be computed, as well as the kernel approximations
\begin{align*}
  k\left(\vec{\xi}_{\alpha}^{\sigma},\vec{\xi}_{\beta}^{\tau}\right)
\end{align*}
for every admissible pair of clusters \(\left(\sigma,\tau\right)\).
We refer the reader to \cite{SauterSchwab2010_BoundaryElementMethods} for further details.

In order for the consistency error of the cluster method to be dominated by the discretization error of the method, the interpolation order \(m\) essentially needs to grow with \(\abs{\log h_{\min}}\).
For further details, we refer to \cite{SauterSchwab2010_BoundaryElementMethods,AinsworthGlusa2017_TowardsEfficientFiniteElement}.

By following the arguments in \cite{SauterSchwab2010_BoundaryElementMethods}, it can be shown that the number of near field entries scales linearly in \(n\).
The same conclusion holds for the number of far field cluster pairs.
Since the four dimensional integral contributions \(a^{K\times \tilde{K}}\) are evaluated using Gaussian quadrature rules with at most \(k=\mathcal{O}\left(\log n\right)\) quadrature nodes per dimension, the assembly of the near field contributions scales with \(n\log^{2d}n\) .
The far field kernel approximations and the shift coefficients have size \(m^{2d}=\mathcal{O}\left(\log^{2d}n\right)\), and are also calculated in \(\log^{2d}n\) complexity.
This means that all the kernel approximations and shift coefficients are obtained in \(n\log^{2d}n\) time.
Finally, the \(nm^{d}\) basis far-field coefficients require the evaluation of integrals using \(m+1\)-th order Gaussian quadrature, leading to a complexity of \(n\log^{2d}n\) as well.
The overall complexity of the cluster method is therefore \(\mathcal{O}\left(n\log^{2d}n\right)\), and the sparse approximation requires \(\mathcal{O}\left(n\log^{2d}n\right)\) memory.
In practice, this means that the assembly of the near-field matrix dominates the other steps but involves only local computations.

The computation of the matrix-vector product involving upward and downward recursion in the cluster tree and multiplication by the kernel approximations can also be shown to scale with \(\mathcal{O}\left(n\log^{2d}n\right)\).

The same principle can be used to evaluate the strong form in all the quadrature nodes necessary for the computation of the BR error indicators.
According to \cref{eq:8,eq:9}, a naive implementation scales as \(\mathcal{O}\left(n^{2}\right)\).
Many of the components that are used for the cluster method in the matrix-vector product can be reused.
Let \(u=\sum_{i\in\mathcal{I}_{h}}u_{i}\phi_{i}\in V_{h}\), and let \(\vec{x}_{0}\in K_{0}\) be a quadrature node in the evaluation of the local \(L^{2}\)-norms in \eqref{eq:BRindicators}.
Moreover, let \(i_{0}\in\mathcal{I}_{h}\) be the closest degree of freedom to \(\vec{x}_{0}\) on \(K_{0}\).

Let \(\left(\sigma,\tau\right)\) be an admissible cluster pair such that \(i_{0}\in\sigma\), and set \(u_{\tau}=\sum_{j\in\tau}u_{j}\phi_{j}\), so that \(\operatorname{supp} u_{\tau}\subset Q_{\tau}\).
The contribution to the fractional Laplacian is approximated by
\begin{align*}
  -\int_{Q_{\tau}} k\left(\vec{x},\vec{y}\right) u_{\tau}\left(\vec{y}\right) \dd\vec{y}
  &= -\int_{Q_{\sigma}} \int_{Q_{\tau}} k\left(\vec{x},\vec{y}\right) \delta\left(\vec{x}-\vec{x}_{0}\right)u_{\tau}\left(\vec{y}\right) \dd\vec{y} \dd{\vec{x}} \\
  &\approx \sum_{\alpha,\beta=1}^{m^{d}} k\left(\vec{\xi}_{\alpha}^{\sigma},\vec{\xi}_{\beta}^{\tau}\right) \int_{Q_{\sigma}} L_{\alpha}^{\sigma}\left(\vec{x}\right)\delta\left(\vec{x}-\vec{x}_{0}\right) \dd\vec{x} \int_{Q_{\tau}} L_{\beta}^{\tau}\left(\vec{y}\right)u_{\tau}\left(\vec{y}\right) \dd\vec{y} \\
  &= \sum_{\alpha,\beta=1}^{m^{d}} k\left(\vec{\xi}_{\alpha}^{\sigma},\vec{\xi}_{\beta}^{\tau}\right) L_{\alpha}^{\sigma}\left(\vec{x}_{0}\right) \int_{Q_{\tau}} L_{\beta}^{\tau}\left(\vec{y}\right)u_{\tau}\left(\vec{y}\right) \dd\vec{y}.
\end{align*}
The integrals \(\int_{Q_{\tau}} L_{\beta}^{\tau}\left(\vec{y}\right)u_{\tau}\left(\vec{y}\right) \dd\vec{y}\) already appeared in the computation of the matrix-vector product, and can be recursively computed.
The same holds for the factors \(L_{\alpha}^{\sigma}\left(\vec{x}_{0}\right)\).

On the other hand, if the cluster pair \(\left(\sigma,\tau\right)\) is not admissible and therefore in the near-field, then formulas \cref{eq:8,eq:9} with \(\Omega\) set to \(Q_{\sigma}\) can be used to evaluate the fractional Laplacian.

This shows that the evaluation of the strong form of the fractional Laplacian on all the quadrature nodes \(\left\{\vec{x}_{j}\right\}_{j}\) can be seen as an application of the cluster method with respect to the sets of functions \(\left\{\delta\left(\cdot-\vec{x}_{j}\right)\right\}_{j}\) and \(\left\{\phi_{i}\right\}_{i}\).

\section{Numerical Examples}
\label{sec:numerical-examples}

For \(s\in(0,1)\), we consider the fractional Poisson problem
\begin{align*}
  \left(-\Delta\right)^{s}u = f &\text{ in }\Omega=B(0,1),\\
  u=0&\text{ in }\Omega^{c}
\end{align*}
in \(d\in\left\{1,2\right\}\) dimensions.
Closed-form solutions to the problem posed on a disc are available.
In \(d=1\) dimensions, the solution when the right-hand side is
\begin{align*}
  % f_{n,0}&= \frac{2^{2s}\Gamma\left(1+s+n\right)\Gamma\left(1/2+s+n\right)}{\Gamma\left(1+n\right)\Gamma\left(1/2+n\right)} P_{n}^{(s,-1/2)}\left(2r^{2}-1\right), &n\geq0
 % \\
  f_{n,0}^{1D}&= 2^{2s}\Gamma\left(1+s\right)^{2}\binom{s+n-1/2}{s}\binom{s+n}{s} P_{n}^{(s,-1/2)}\left(2x^{2}-1\right), &n\geq0
\end{align*}
is given by
\begin{align*}
  u_{n,0}^{1D}&= P_{n}^{(s,-1/2)}\left(2x^{2}-1\right) \left(1-x^{2}\right)^{s}_{+},
\end{align*}
where \(\binom{x}{y}=\frac{\Gamma\left(x+1\right)}{\Gamma\left(y+1\right)\Gamma\left(x-y+1\right)}\) are generalised binomial coefficients, \(P_{n}^{(\alpha,\beta)}\) are the Jacobi polynomials, and \(x_{+}=\max\{0,x\}\).
Moreover, when the right-hand side is given by
\begin{align*}
  % f_{n,1}&= \frac{2^{2s}\Gamma\left(1+s+n\right)\Gamma\left(3/2+s+n\right)}{\Gamma\left(1+n\right)\Gamma\left(3/2+n\right)} x P_{n}^{(s,1/2)}\left(2r^{2}-1\right), &n\geq0
 % \\
  f_{n,1}^{1D}&= 2^{2s}\Gamma\left(1+s\right)^{2}\binom{s+n+1/2}{s}\binom{s+n}{s} x P_{n}^{(s,1/2)}\left(2x^{2}-1\right), &n\geq0
\end{align*}
then the solution is
\begin{align*}
  u_{n,1}^{1D}&= x P_{n}^{(s,1/2)}\left(2x^{2}-1\right) \left(1-x^{2}\right)^{s}_{+}
\end{align*}
as shown in \cite{DydaKuznetsovEtAl2016_FractionalLaplaceOperatorMeijerGFunction}.
Turning to \(d=2\) dimensions, the solution when the right-hand side is
\begin{align*}
  % f_{n,\ell}
  % &= \frac{2^{2s}\Gamma\left(1+s+n\right)\Gamma\left(1+s+n+\ell\right)}{\Gamma\left(1+n\right)\Gamma\left(1+n+\ell\right)} r^{\ell}\cos\left(\ell\theta\right) P_{n}^{(s,\ell)}\left(2r^{2}-1\right),
  % &\ell,n\geq0 \\
  f_{n,\ell}^{2D}
  &= 2^{2s}\Gamma\left(1+s\right)^{2}\binom{s+n+\ell}{s}\binom{s+n}{s} r^{\ell}\cos\left(\ell\theta\right) P_{n}^{(s,\ell)}\left(2r^{2}-1\right),
  &\ell,n\geq0
\end{align*}
is given by
\begin{align*}
  u_{n,\ell}^{2D} = r^{\ell}\cos\left(\ell\theta\right) P_{n}^{(s,\ell)}\left(2r^{2}-1\right)\left(1-r^{2}\right)^{s}_{+}.
\end{align*}
For more details on how these analytical solutions are determined, see \cite{DydaKuznetsovEtAl2016_FractionalLaplaceOperatorMeijerGFunction}.
Observe that in both the one-dimensional and the two-dimensional cases, the solutions contains a term which behaves like \(\delta\left(\vec{x}\right)^{s}\), where \(\delta\left(\vec{x}\right)\) is the distance from \(\vec{x}\in\Omega\) to the boundary.

\subsection{One-Dimensional Examples}
\label{sec:one-dimens-example}

We first solve the fractional Poisson problem with constant right-hand side \(f=f_{0,0}^{1D}\) and solution \(u=u_{0,0}^{1D}\) for \(s=0.25\) and \(s=0.75\).
According to the regularity results in \Cref{thm:regularity}, \(u\) is in \(H^{s+1/2-\varepsilon}\left(\Omega\right)\), so that we expect essentially \(h^{1/2}=\mathcal{O}\left(n^{-1/2}\right)\) convergence if a globally quasi-uniform mesh is used.
Using graded meshes, the expected rate in \(d=1\) dimensions is \(n^{s-2}\), as shown in \cref{eq:predictedRate1D}.

We adaptively refine the mesh based on the BR error indicators and a maximum marking strategy with threshold \(0.8\).
The cluster method is used both for the stiffness matrix and for the evaluation of the strong form of the fractional Laplacian.
The error plots in \Cref{fig:errorAdaptive1d} clearly display \(n^{s-2}\) convergence in \(\widetilde{H}^{s}\)-norm, both for \(s=0.25\) and for \(s=0.75\), as compared to the \(n^{-1/2}\) convergence obtained by uniform mesh refinement.
This means that the optimal rate of convergence for \(d=1\) is achieved.
The measured \(L^{2}\)-convergence is of order \(n^{-1/2-s}\) for uniform refinement and \(n^{-2}\) for adaptive refinement for both values of \(s\).

\begin{figure}
  \centering
  \includegraphics{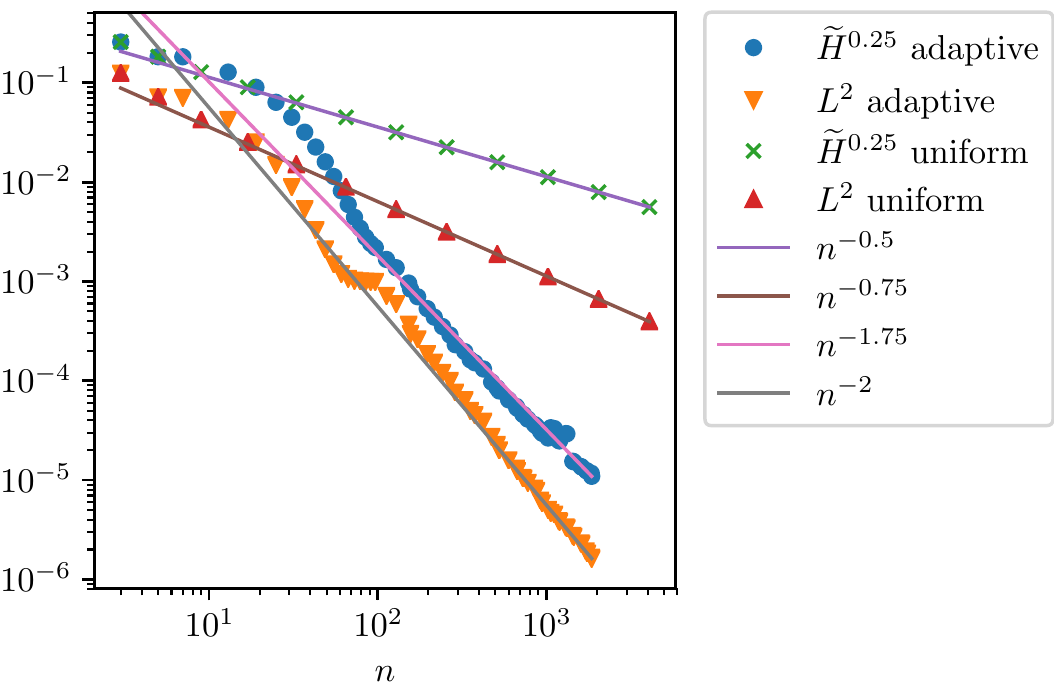}
  \includegraphics{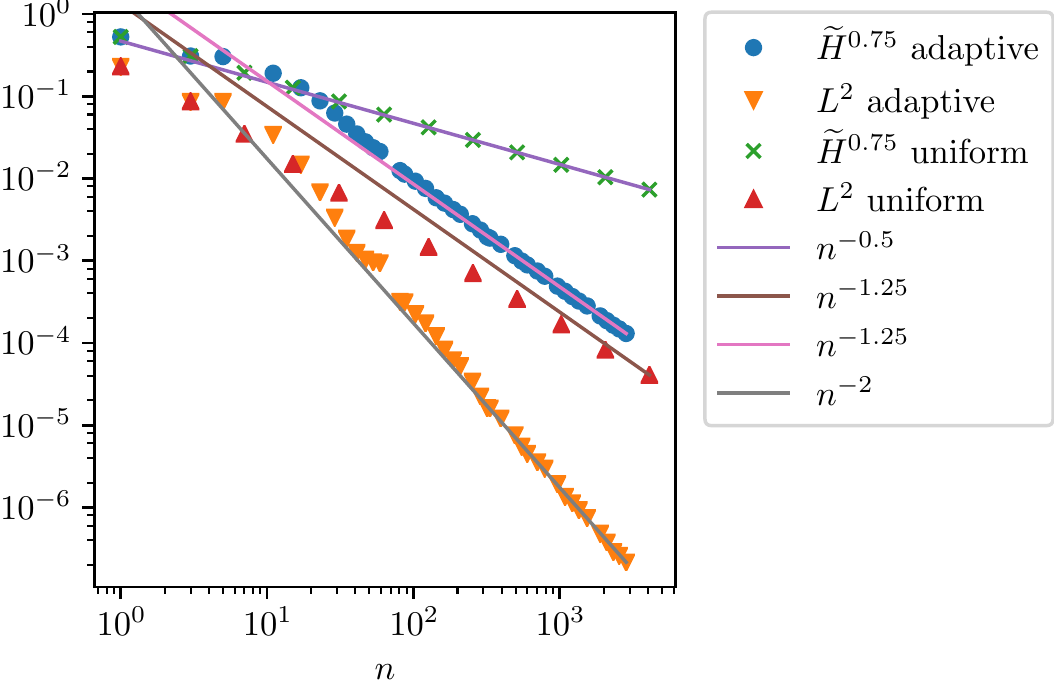}
  \caption{
    \(\widetilde{H}^{s}\)- and \(L^{s}\)-error for the one-dimensional problem with constant right-hand side \(f=f_{0,0}^{1D}\) and for uniform and adaptive refinement.
    % \(s=0.25\) on the left, \(s=0.75\) on the right.
    \(s=0.25\) at the top, \(s=0.75\) at the bottom.
    It can be seen that while uniform refinement results in convergence in \(\widetilde{H}^{s}\)-norm with rate \(n^{-1/2}\), adaptive refinement achieves the optimal rate of \(n^{s-2}\).
    The convergence in \(L^{2}\)-norm is of order \(n^{-1/2-s}\) for uniform refinement, and of order \(n^{-2}\) for adaptive refinement.
  }
  \label{fig:errorAdaptive1d}
\end{figure}

In \Cref{fig:timingsAdaptive1d}, we plot the number of degrees of freedom \(n\) against timing results for the assembly of the stiffness matrix, the solution of the linear system, and the computation of the error indicators \eqref{eq:BRindicators}.
It can be observed that all three scale with \(n\left(\log n\right)^{2}\).

\begin{figure}
  \centering
  \includegraphics{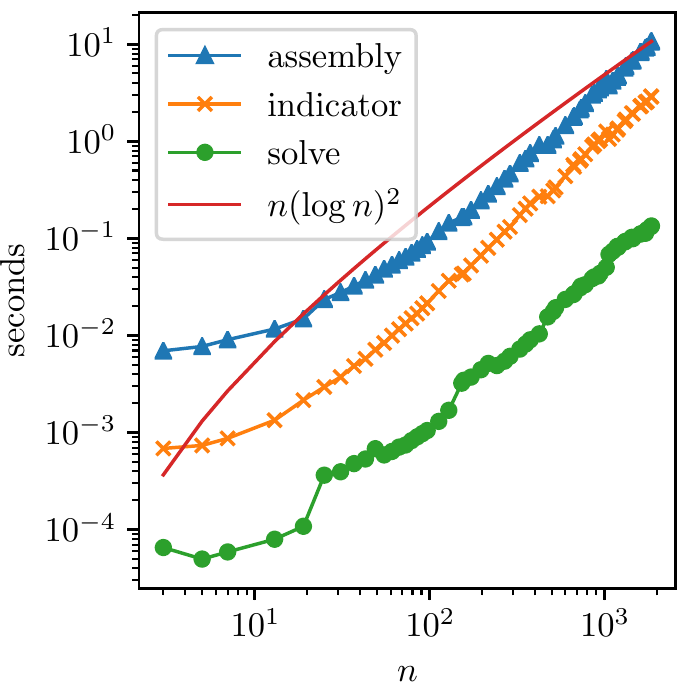}
  \includegraphics{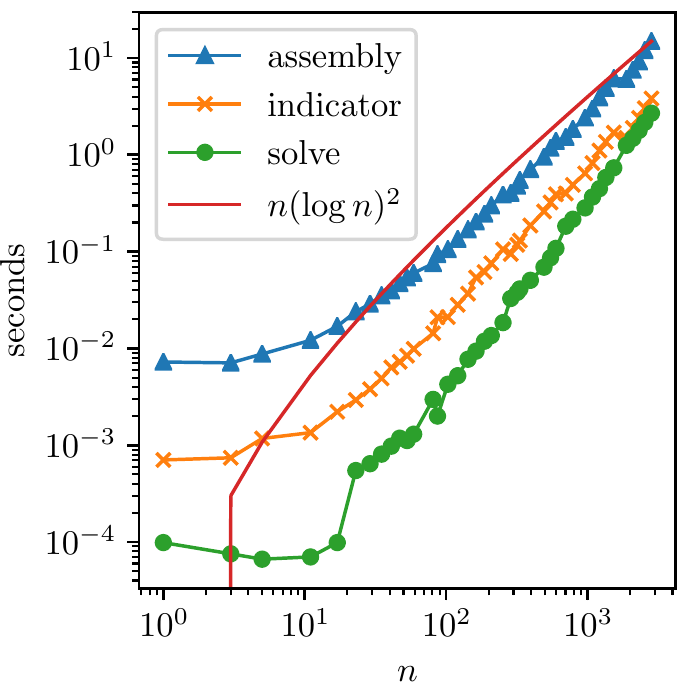}
  \caption{
    Timings for assembly of the stiffness matrix, solution of linear system and computation of the error indicators for the one-dimensional problem with constant right-hand side \(f=f_{0,0}^{1D}\).
    \(s=0.25\) on the left, \(s=0.75\) on the right.
   All steps can be seen to scale quasi-linearly with the number of unknowns \(n\).
  }
  \label{fig:timingsAdaptive1d}
\end{figure}

Next, we consider the solution of the problem with right-hand side \(f(x)=\operatorname{sign} x\in H^{1/2-\varepsilon}\left(\Omega\right)\).
Based on decomposition of \(f\) with respect to \(f_{n,\ell}^{1D}\), the analytic solution \(u\) is found to be
\begin{align*}
  u
  % &= 2^{-2s}\frac{\left(1-r^{2}\right)^{s}_{+}}{\Gamma\left(1+s\right)^{2}}\left\{
    % \frac{1}{2} + \sum_{n\geq0,~\ell\geq1 \text{ and odd}} \frac{(-1)^{(\ell+1)/2+n}(2n+s+l+1)}{\pi(n+l/2)(s+1)} \frac{\cos\left(\ell\theta\right) r^{\ell} P_{n}^{(s,\ell)}\left(2r^{2}-1\right)}{\binom{n+s+l/2+1}{n+l/2}\binom{s+n}{n}}
    % \right\}, \\
  &=2^{-2s}\sum_{n\geq0} (-1)^{n}\frac{2n+s+3/2}{n+1/2} \frac{\binom{n+s+1/2}{-1/2}}{\binom{n+s+1/2}{s}\binom{n+s}{s}} x P_{n}\left(2x^{2}-1\right) \left(1-x^{2}\right)_{+}^{s}
\end{align*}
so that the discretization error can be calculated as
\begin{align}
  \norm{u-u_{h}}_{\V}^{2}=a\left(u-u_{h},u-u_{h}\right) = \left(f,u\right) - \left(f,u_{h}\right),\label{eq:errorEq1d}
\end{align}
with
\begin{align*}
  \left(f,u\right)
  % &=\frac{2^{-2s}}{\pi \Gamma\left(s+5/2\right)^{2}}\sum_{n\geq0}\frac{2n+s+3/2}{\binom{n+s+1}{n-1/2}^{2}}.
  &= \frac{2^{1-2s}}{(2s+1)\Gamma\left(1+s\right)^{2}}
\end{align*}

Solutions for \(s=i/10\), \(i=1,\dots,9\) are plotted in \Cref{fig:solutions1D}.
For smaller values of \(s\), the transition at the boundary and at \(x=0\) is more pronounced, whereas for larger values of \(s\) the solutions resemble the solution of the standard integer-order Poisson problem.

\begin{figure}
  \centering
  \includegraphics{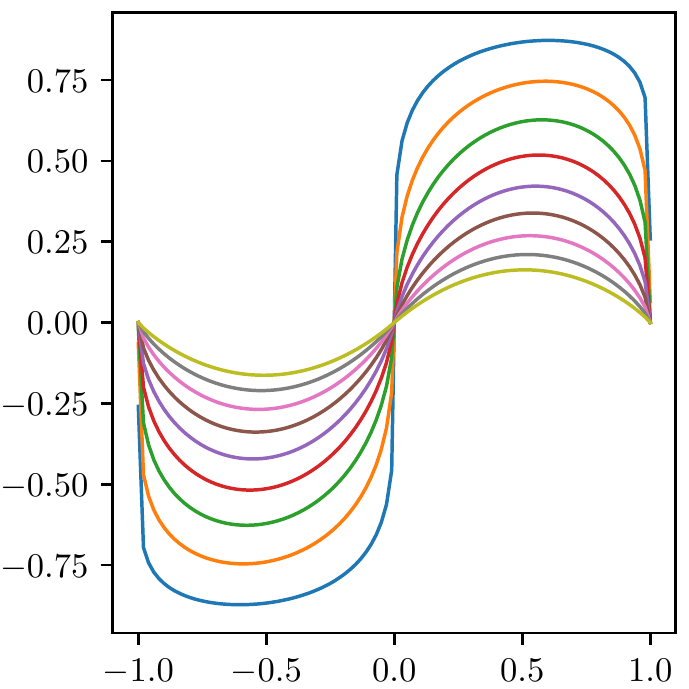}
  \caption{
    Solutions to the one-dimensional fractional Poisson problem with right-hand side \(f\left(x\right)=\operatorname{sign}x\) for \(s=i/10\), \(i=1,\dots,9\).
    For smaller values of \(s\), the transition at the boundary and at \(x=0\) is more pronounced, whereas for larger values of \(s\) the solutions resemble the solution of the standard integer-order Poisson problem.
}
  \label{fig:solutions1D}
\end{figure}

The error plots in \Cref{fig:errorAdaptive1dplateau} show that \(n^{s-2}\)-convergence in \(\widetilde{H}^{s}\)-norm is also achieved for the discontinuous right-hand side \(f(x)=\operatorname{sign}x\).
\begin{figure}
  \centering
  \includegraphics{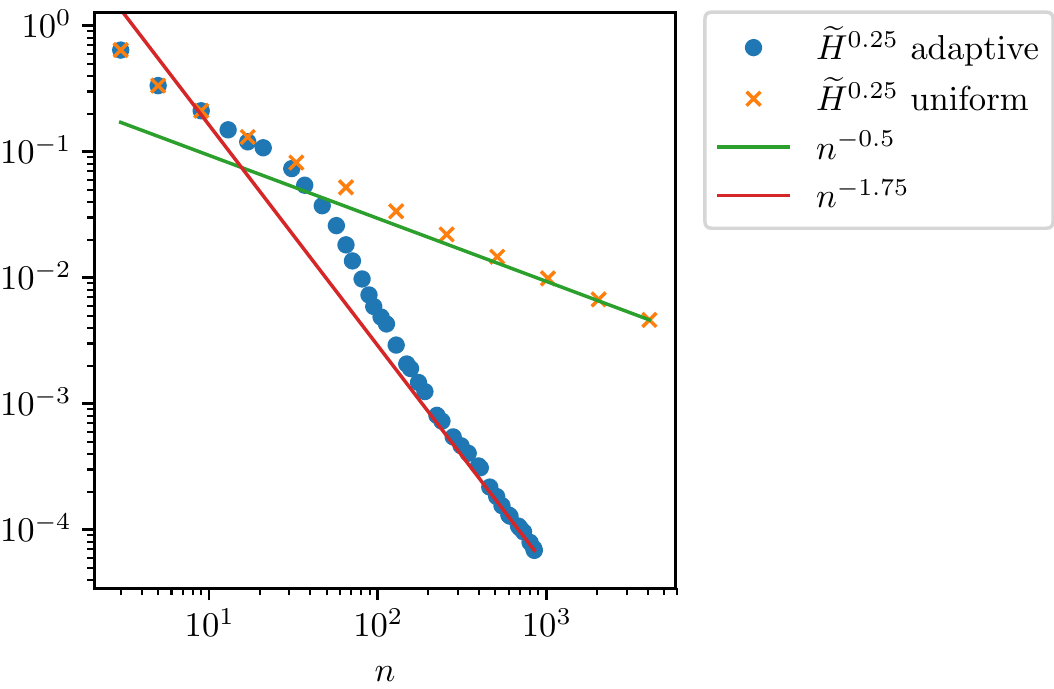}
  \includegraphics{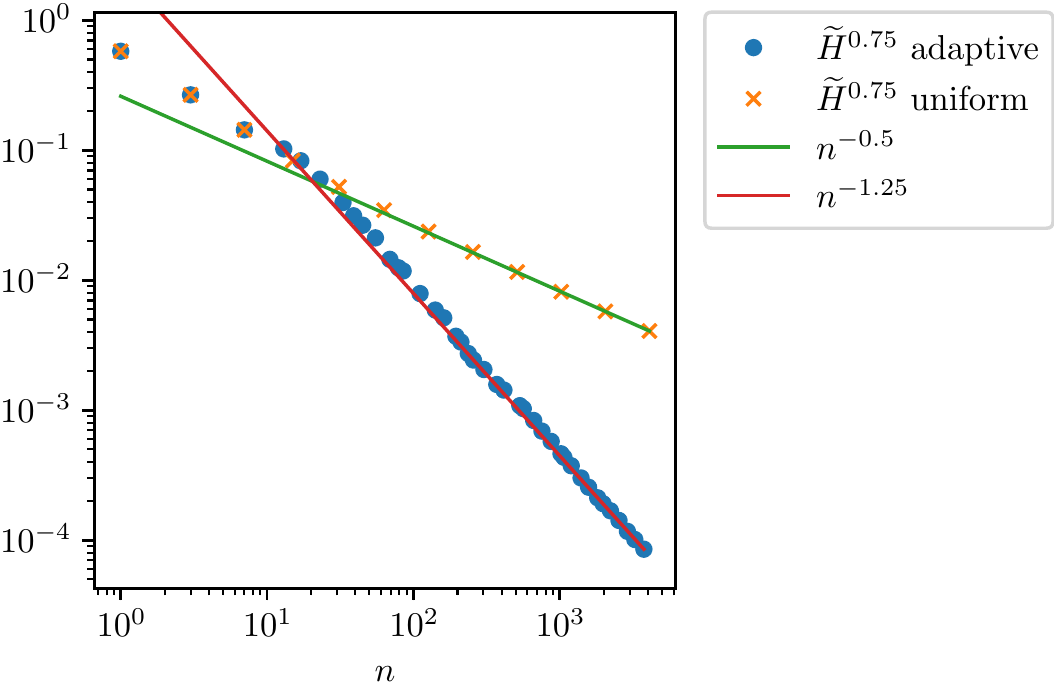}
  \caption{
    \(\widetilde{H}^{s}\)-error for the one-dimensional problem with discontinuous right-hand side \(f(x)=\operatorname{sign}x\) and for uniform and adaptive refinement.
    % \(s=0.25\) on the left, \(s=0.75\) on the right.
    \(s=0.25\) at the top, \(s=0.75\) at the bottom.
    It can be seen that while uniform refinement results in convergence in \(\widetilde{H}^{s}\)-norm with rate \(n^{-1/2}\), adaptive refinement achieves the optimal rate of \(n^{s-2}\).
  }
  \label{fig:errorAdaptive1dplateau}
\end{figure}

\subsection{Two-Dimensional Examples on the Unit Disc}
\label{sec:two-dimens-example}

We consider the two-dimensional fractional Poisson problem on a disc with constant right-hand side \(f=f_{0,0}^{2D}\).
The solutions \(u\) for \(s=0.25\) and \(s=0.75\) are shown in \Cref{fig:solutions}.

Using a globally quasi-uniform mesh, we would obtain \(h^{1/2}=\mathcal{O}\left(n^{-1/4}\right)\) convergence.
Using a graded mesh, \(n^{-1/2}\) convergence was predicted, owing to the fact that the mesh is required to be locally quasi-uniform.
Again, we adaptively refine the mesh, based on the error indicators \eqref{eq:BRindicators}.
The obtained meshes (\Cref{fig:meshes}) are highly refined in regions close to the boundary, where the solution is singular.
\begin{figure}
  \centering
  \includegraphics{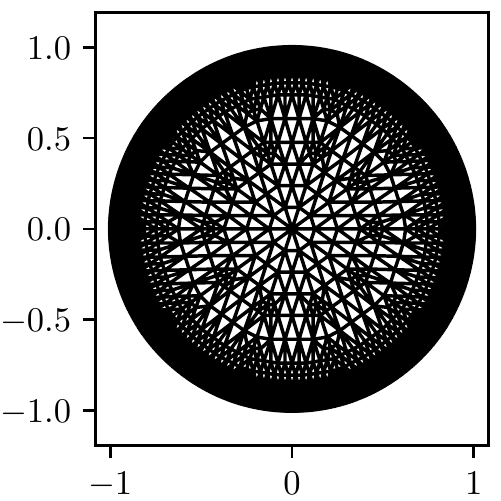}
  \includegraphics{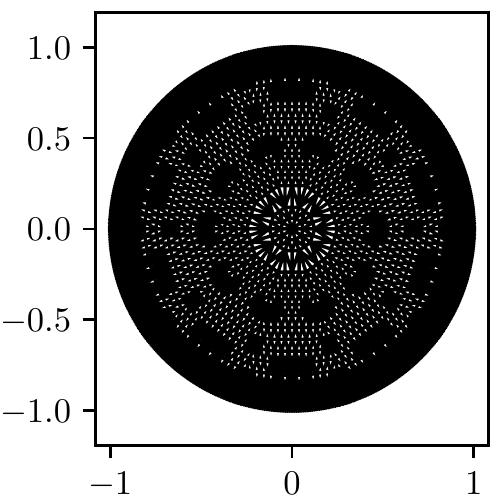}
  \caption{
    Adaptively refines meshes for the two-dimensional problem with constant right-hand side.
    \(s=0.25\) on the left, \(s=0.75\) on the right.}
  \label{fig:meshes}
\end{figure}

We plot the \(\V\)-error in \Cref{fig:errorAdaptive2d}.
It can be observed that \(n^{-1/2}\) convergence is obtained, thus matching the optimal rate obtained using a graded mesh for a constant right-hand side in \eqref{eq:predictedRate2D}.
As in the one-dimensional case, the rate in \(L^{2}\)-norm is \(s\) orders higher than the rate in \(\widetilde{H}^{s}\)-norm.
\begin{figure}
  \centering
  \includegraphics{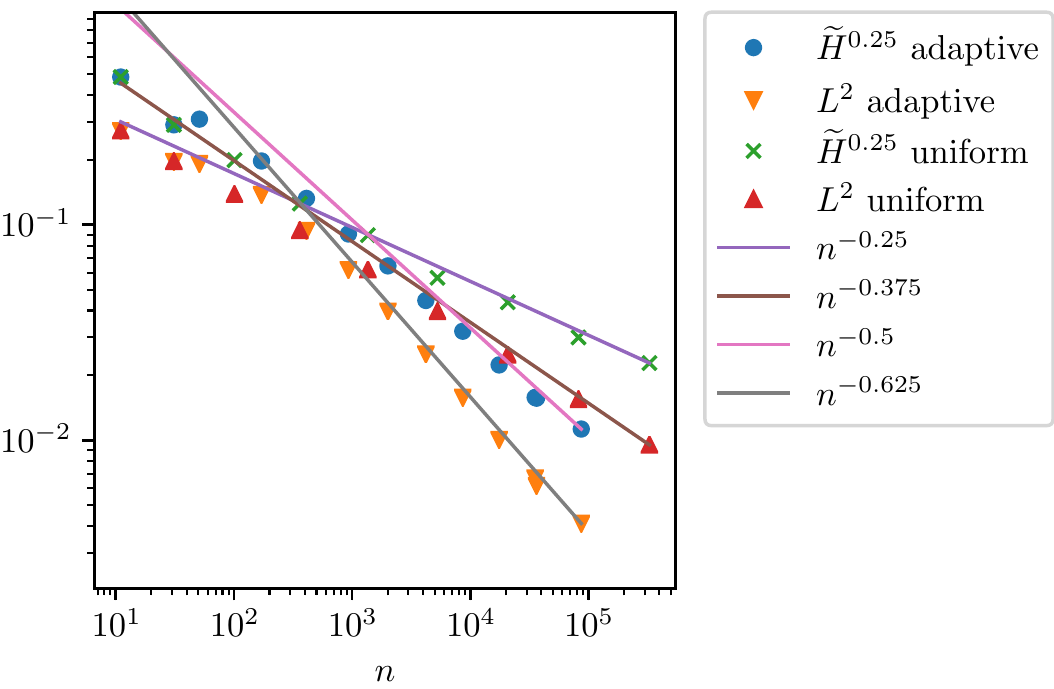}
  \includegraphics{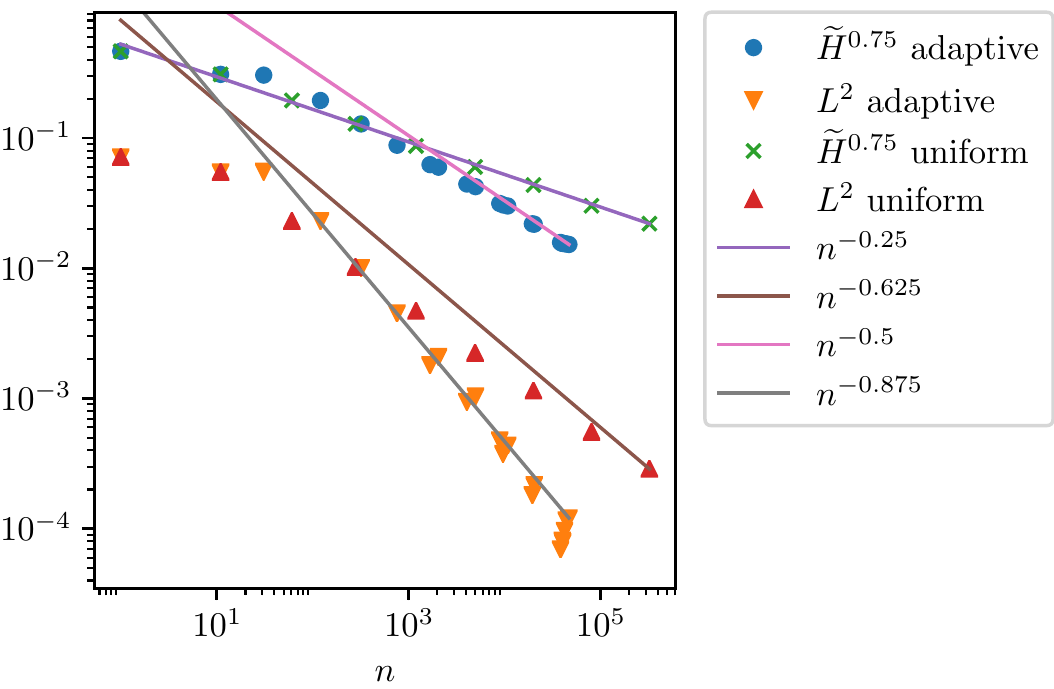}
  \caption{
  \(\widetilde{H}^{s}\)- and \(L^{s}\)-error for the two-dimensional problem with constant right-hand side \(f=f_{0,0}^{2D}\) and for uniform and adaptive refinement.
    % \(s=0.25\) on the left, \(s=0.75\) on the right.
    \(s=0.25\) at the top, \(s=0.75\) at the bottom.
    It can be seen that while uniform refinement results in convergence in \(\widetilde{H}^{s}\)-norm with rate \(n^{-1/4}\), adaptive refinement achieves the optimal rate of \(n^{-1/2}\).
    The convergence in \(L^{2}\)-norm is of order \(n^{-1/4-s/2}\) for uniform refinement, and of order \(n^{-1/2-s/2}\) for adaptive refinement.}
  \label{fig:errorAdaptive2d}
\end{figure}
By considering the timing results shown in \Cref{fig:timingsAdaptive2d}, we see that all phases of the adaptive refinement loop scale with \(n \left(\log n\right)^{4}\).
The computation of the error indicators is cheaper than the assembly of the system matrix, and the solution of the linear system is significantly cheaper than both.
Relative to the matrix assembly, the computation of the error indicators is more expensive than for the one dimensional problem, since it involves numerical quadrature over element edges instead of pointwise evaluations.
\begin{figure}
  \centering
  \includegraphics{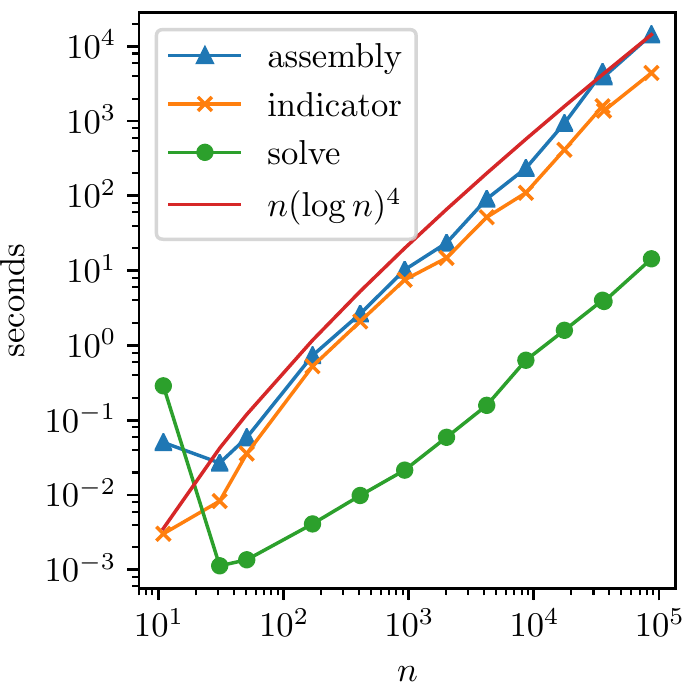}
  \includegraphics{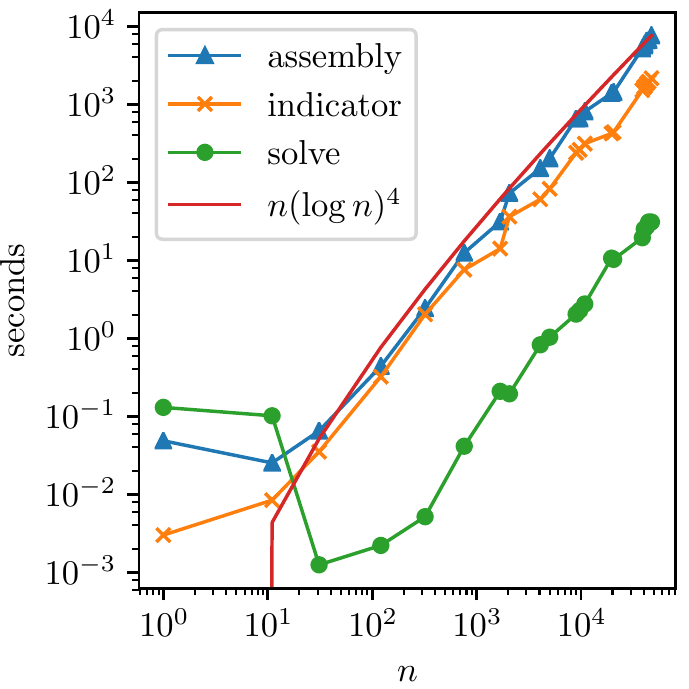}
  \caption{
    Timings for assembly of the stiffness matrix, solution of linear system and computation of the error indicators for the two-dimensional problem with constant right-hand side \(f=f_{0,0}^{2D}\).
    \(s=0.25\) on the left, \(s=0.75\) on the right.
   All steps can be seen to scale quasi-linearly with the number of unknowns \(n\).}
  \label{fig:timingsAdaptive2d}
\end{figure}

We now consider the solution of the problem with right-hand side \(f(r,\theta)=1_{\abs{\theta}<\pi/2}\in H^{1/2-\varepsilon}\left(\Omega\right)\).
Based on decomposition of \(f\) with respect to \(f_{n,\ell}^{2D}\), the analytic solution \(u\) is found to be
\begin{align*}
  u
  &= 2^{-2s}\frac{\left(1-r^{2}\right)^{s}_{+}}{\Gamma\left(1+s\right)^{2}}\left\{
    \frac{1}{2}  \right.\\
  &\qquad \left.+\sum_{n\geq0,~\ell\geq1 \text{ and odd}} \frac{(-1)^{(\ell+1)/2+n}(2n+s+l+1)}{\pi(n+l/2)(s+1)} \frac{\cos\left(\ell\theta\right) r^{\ell} P_{n}^{(s,\ell)}\left(2r^{2}-1\right)}{\binom{n+s+l/2+1}{n+l/2}\binom{s+n}{n}}
    \right\},
\end{align*}
so that the discretization error can be calculated as
\begin{align*}
  \norm{u-u_{h}}_{\V}^{2}=a\left(u-u_{h},u-u_{h}\right) = \left(f,u\right) - \left(f,u_{h}\right),
\end{align*}
with
\begin{align*}
  % \left(f,u\right)
  % &= 2^{-2s}\left\{\frac{\pi}{4(s+1)\Gamma\left(s+1\right)^{2}} + \frac{1}{2\pi\Gamma\left(s+3\right)^{2}} \sum_{n\geq0,~\ell\geq1 \text{ and odd}}\frac{2n+s+\ell+1}{\binom{n+s+\ell/2+1}{s+2}^{2}}\right\}. \\
  % \left(f,u\right)
  % &= 2^{-2s}\left\{\frac{\pi}{4(s+1)\Gamma\left(s+1\right)^{2}} + \frac{1}{2\pi\Gamma\left(s+3\right)^{2}} \sum_{k\geq0}\frac{(2k+s+2)(k+1)}{\binom{k+s+3/2}{s+2}^{2}}\right\}. \\
    % &=2^{-2s}\left\{\frac{\pi}{4(s+1)\Gamma\left(s+1\right)^{2}} + \frac{s+2}{2\pi\Gamma\left(s+3\right)^{2}\binom{s+3/2}{s+2}} {}_{4}F_{3}\left(3,2+s/2,1/2,1/2;5/2+s,5/2+s,1+s/2;1\right)\right\}. \\
  \left(f,u\right)
    &= 2^{-2s}\left\{\frac{\pi}{4(s+1)\Gamma\left(s+1\right)^{2}}  \right.\\
      &\qquad\left.-\frac{1}{\pi} G_{4,4}^{3,2}\left(
      \left.
    \begin{array}{cccc}
      1,& 1+s/2,&5/2+s,&5/2+s \\
        2,&1/2,&1/2,&2+s/2
    \end{array}
            \right|     -1
\right)\right\},
\end{align*}
% The terms in the sum decay as \(\mathcal{O}\left(k^{-2-2s}\right)\), and hence the partial sums with \(k\leq K\) converge with rate \(K^{-1-2s}\).
where \(G\) is the Meijer G-function \cite{MeijerG}.

The solutions \(u\) for \(s=0.25\) and \(s=0.75\) are shown in \Cref{fig:solutionsPlateau}.
\begin{figure}
  \centering
  % \printPageSize
  % \includegraphics{PDF/{solution-s=0.25-n=2-l=2}.pdf}
  % \includegraphics{PDF/{solution-notPeriodic-s=0.25}.pdf}
  \includegraphics{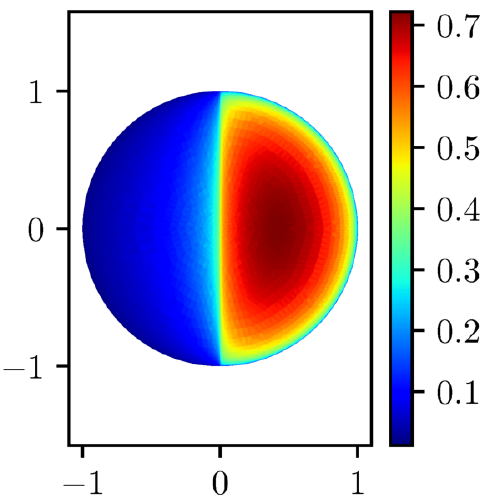}
  \includegraphics{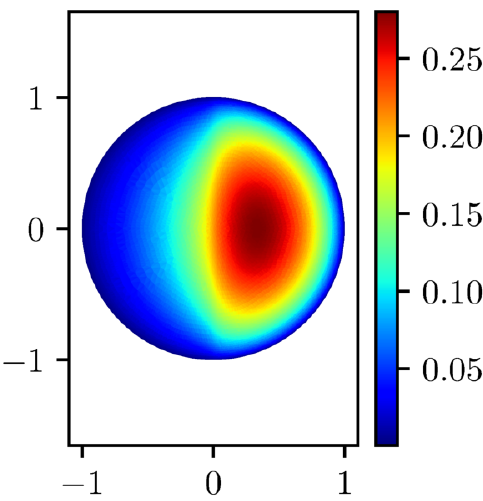}
  \caption{Solutions \(u\) corresponding to the discontinuous right-hand side \(f(r,\theta)=1_{\abs{\theta}<\pi/2}\) for  \(s=0.25\) and for \(s=0.75\).}
  \label{fig:solutionsPlateau}
\end{figure}
The obtained meshes (\Cref{fig:meshesPlateau}) are highly refined in regions close to the boundary and along the discontinuity \(\theta=\pm \pi/2\).
\begin{figure}
  \centering
  \includegraphics{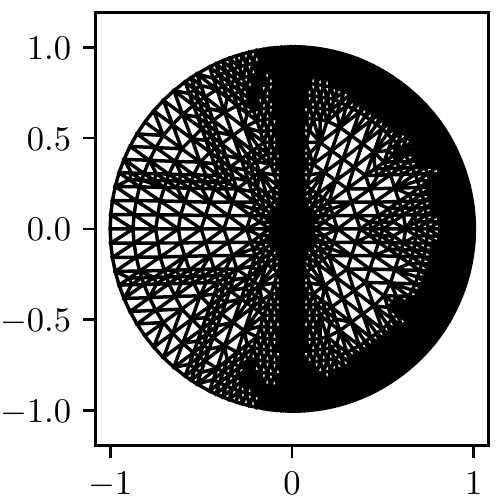}
  \includegraphics{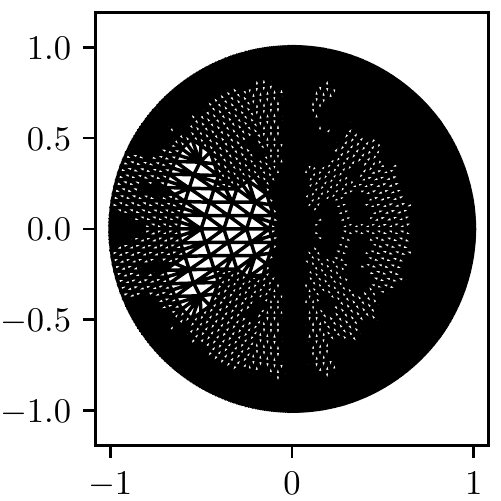}
  \caption{
    Adaptively refined meshes for the discontinuous right-hand side.
    \(s=0.25\) on the left, \(s=0.75\) on the right.}
  \label{fig:meshesPlateau}
\end{figure}
We plot the \(\widetilde{H}^{s}\)-error in \Cref{fig:errorAdaptive2dplateau}, and observe that the optimal rate of convergence of \(n^{-1/2}\) is achieved by adaptive refinement.
\begin{figure}
  \centering
  \includegraphics{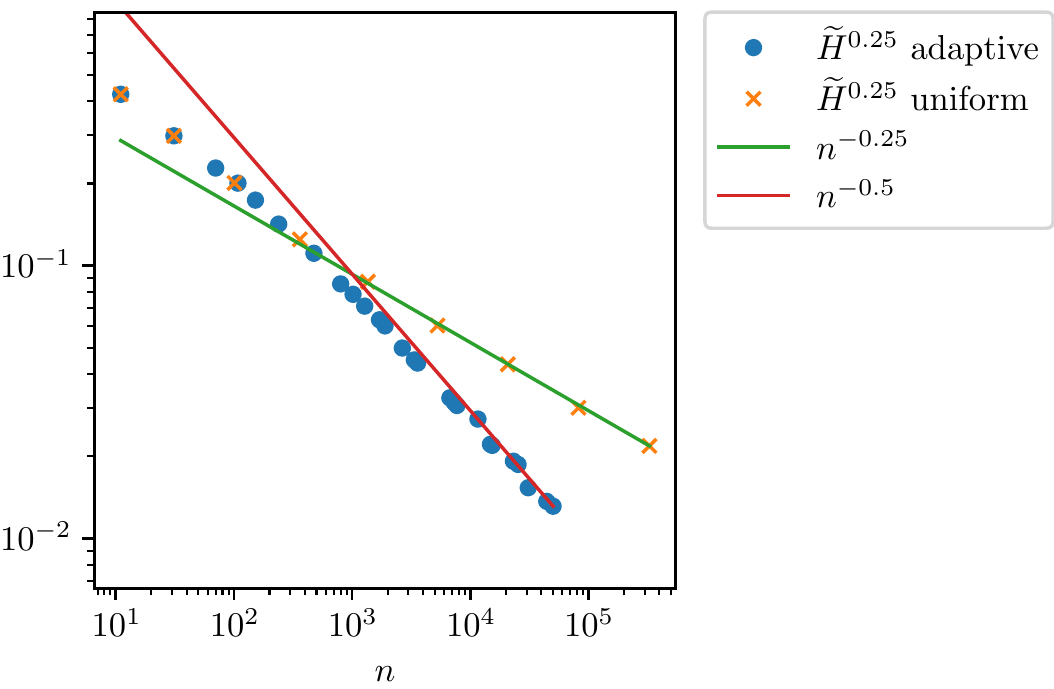}
  \includegraphics{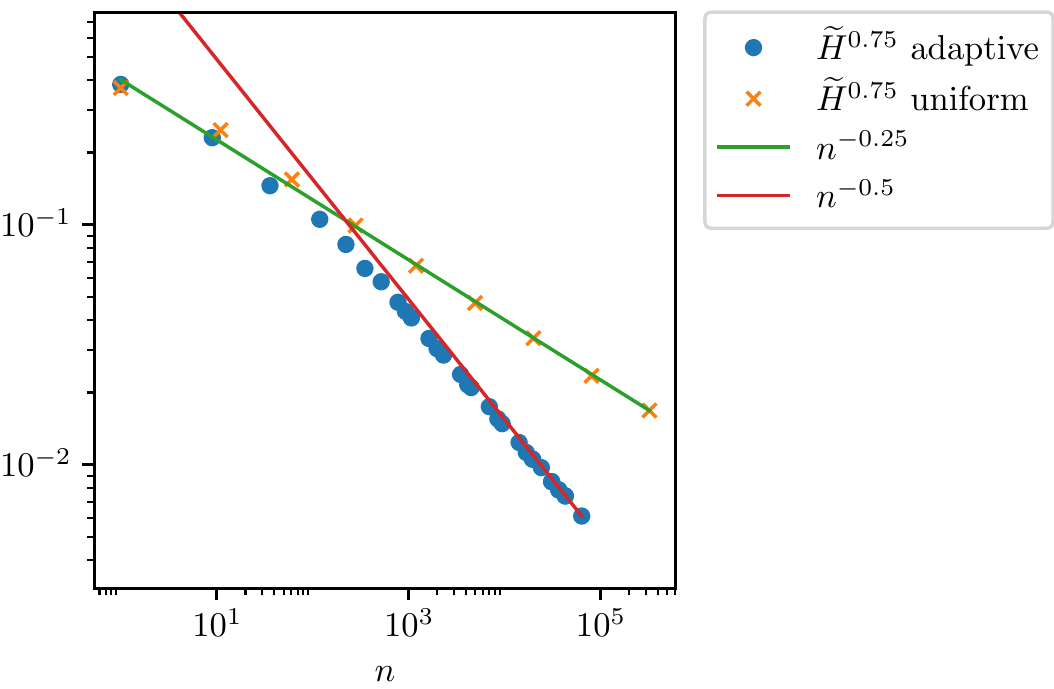}
  \caption{
    \(\widetilde{H}^{s}\)-error for the two-dimensional problem with discontinuous right-hand side \(f(r,\theta)=1_{\abs{\theta}<\pi/2}\) and for uniform and adaptive refinement.
    % \(s=0.25\) on the left, \(s=0.75\) on the right.
    \(s=0.25\) at the top, \(s=0.75\) at the bottom.
    Uniform refinement results in convergence in \(\widetilde{H}^{s}\)-norm with rate \(n^{-1/4}\), adaptive refinement achieves the optimal rate of \(n^{-1/2}\).
  }
  \label{fig:errorAdaptive2dplateau}
\end{figure}

\subsection{L-Shaped Domain}
\label{sec:l-shaped-domain}

We propose to solve the fractional Poisson problem on the L-shaped domain \(\Omega=[0,2]^{2}\setminus [1,2]^{2}\) with constant right-hand side.
Since no analytic solution is available, we solve the problem on a highly refined mesh and use this reference solution \(u_{\underline{h}}\) to approximate the error.
In \Cref{fig:errorLshape}, we show \(\widetilde{H}^{s}\)- and \(L^{2}\) error, as well as the a posteriori error estimate \(\hat{\eta}=\left(\sum_{i}\hat{\eta}_{i}^{2}\right)^{1/2}\), which is based on the local a posteriori error estimators \(\hat{\eta}_{i}\) given in \eqref{eq:BRindicators}.
It can be observed that \(\norm{u_{h}-u_{\underline{h}}}_{\widetilde{H}^{s}}\) and the a posteriori error estimate \(\eta\) converge with optimal rate \(n^{-1/2}\), and that \(\norm{u_{h}-u_{\underline{h}}}_{L^{2}}\) converges with rate \(n^{-1/2-s/2}\).
The apparent speed-up of convergence in \(\widetilde{H}^{s}\)- and \(L^{2}\)-norm for larger number of unknowns is due to the fact that the reference solution \(u_{\underline{h}}\) is used in their computation instead of the true solution \(u\).
The a posteriori error estimate is obviously not suffer this deficiency, because it does not involve the reference solution \(u_{\underline{h}}\).
This example demonstrates that more complicated domains can easily be considered in the presented framework, and that convergence of optimal order can equally well be achieved.

\begin{figure}
  \centering
  % \printPageSize
  \includegraphics{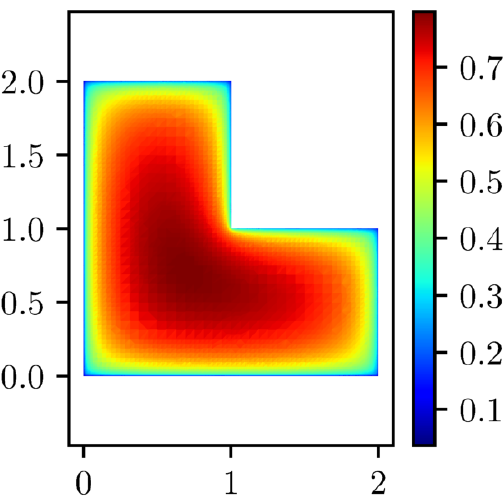}
  \includegraphics{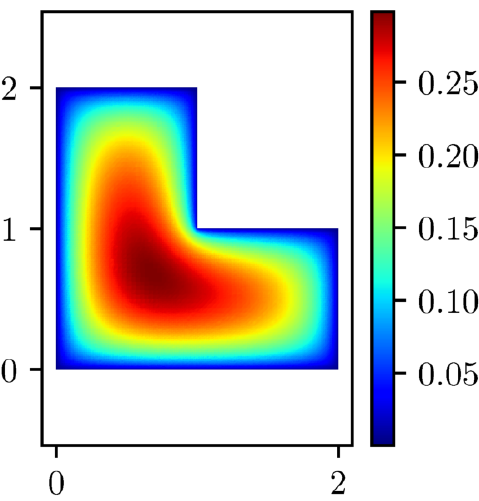}
  \caption{Solutions \(u\) on the L-shaped domain corresponding to a constant right-hand side for \(s=0.25\) (\emph{left}) and for \(s=0.75\) (\emph{right}).}
  \label{fig:solutionsLshape}
\end{figure}

\begin{figure}
  \centering
  \includegraphics{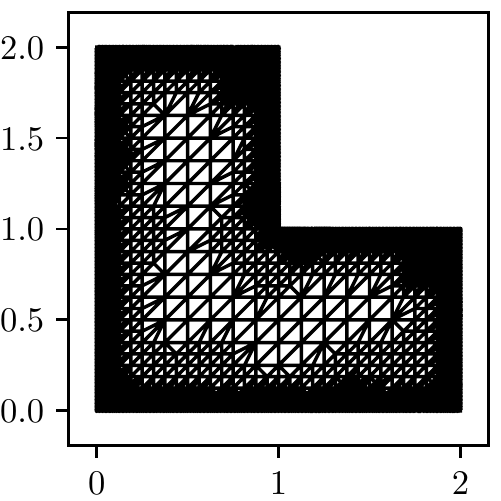}
  \includegraphics{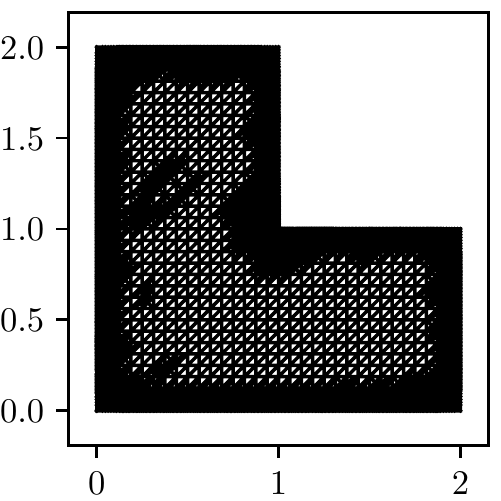}
  \caption{
    Adaptively refined meshes for the L-shaped domain.
    \(s=0.25\) on the left, \(s=0.75\) on the right.}
  \label{fig:meshesLshape}
\end{figure}

\begin{figure}
  \centering
  \includegraphics{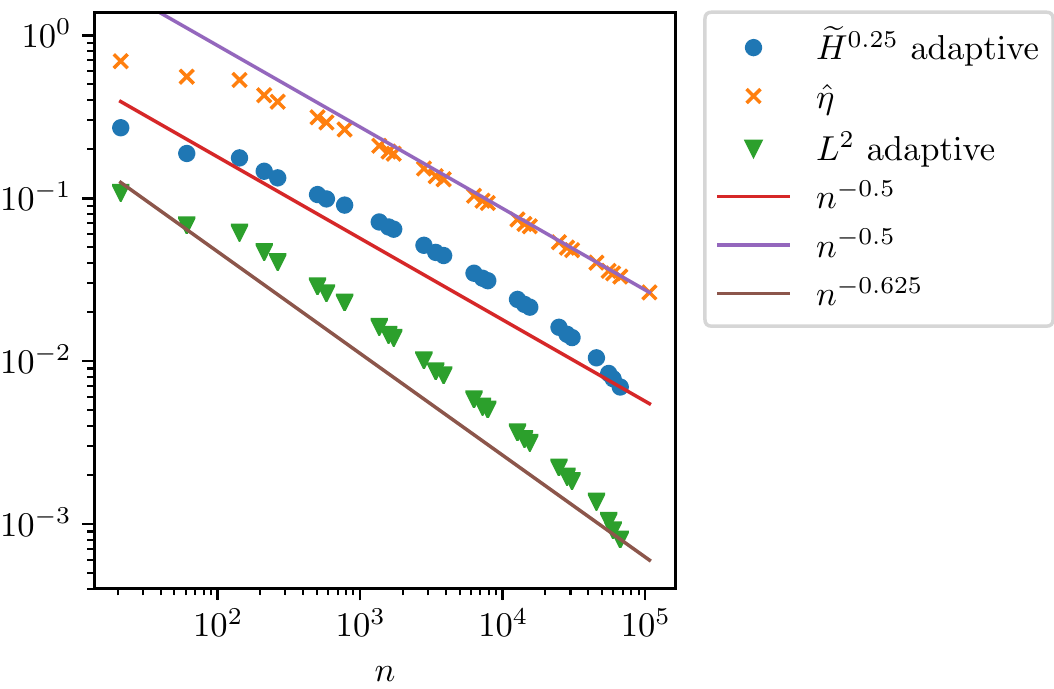}
  \includegraphics{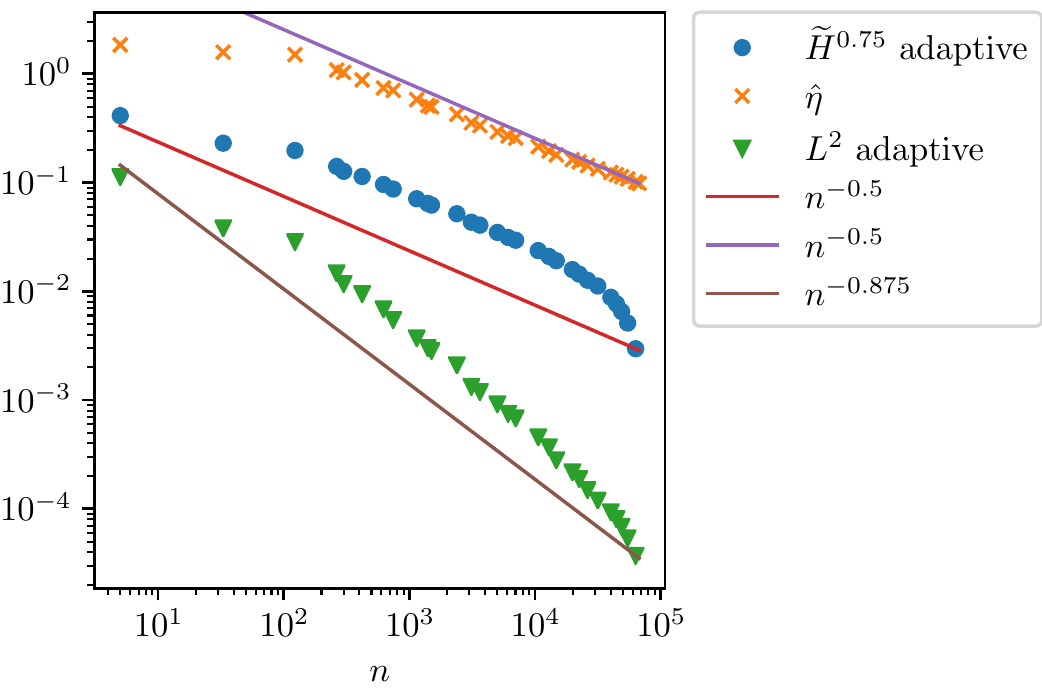}
  \caption{
    \(\widetilde{H}^{s}\)-, \(L^{2}\)-error and estimated error \(\eta\) for the L-shaped domain.
    % \(s=0.25\) on the left, \(s=0.75\) on the right.
    \(s=0.25\) at the top, \(s=0.75\) at the bottom.
    Adaptive refinement results in convergence in \(\widetilde{H}^{s}\)-norm with optimal rate of \(n^{-1/2}\), and the same holds for the a posteriori error \(\hat{\eta}\).
    In \(L^{2}\)-norm, the rate \(n^{-1/2-s/2}\) is observed.
  }
  \label{fig:errorLshape}
\end{figure}

\subsection{A Domain with Disconnected Components}
\label{sec:doma-with-disc}

So far, all the considered domains \(\Omega\) were simply connected.
In fact, this is not a necessary requirement.
Take \(\Omega=\left\{(x,y)\in[0,1]^{2} \mid 0.05 \leq \abs{y-1/2}\right\}\) to be a square domain with a strip removed in the middle.
Then \(\Omega\) has two disconnected components.
We take the right-hand side \(f\) to be one in the upper part of the domain, and zero in the lower part.
Under these conditions, the solution of the regular, integer order Poisson problem is non-zero only on the upper part.
In the fractional case, however, due to the non-locality of the fractional Laplacian, the solution is positive in the interior of the
whole domain, as can be seen in \Cref{fig:solutionsDisconnected}.
The corresponding adaptively refined meshes are shown in \Cref{fig:meshesDisconnected}.

\begin{figure}
  \centering
  % \printPageSize
  \includegraphics{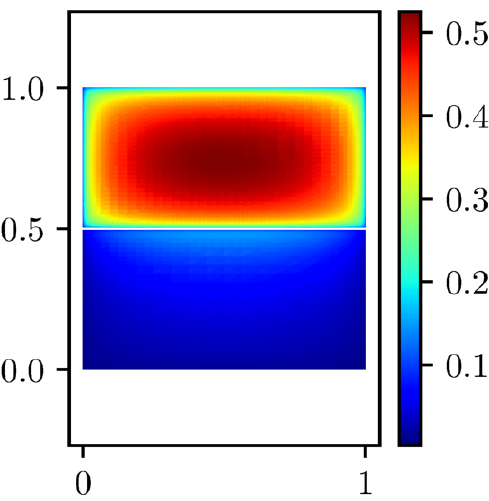}
  \includegraphics{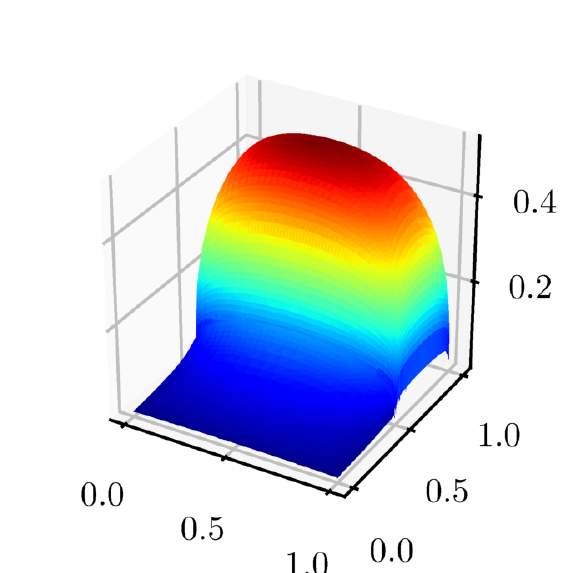}
  \includegraphics{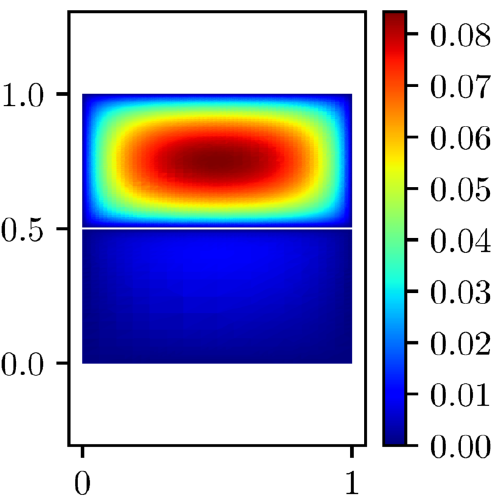}
  \includegraphics{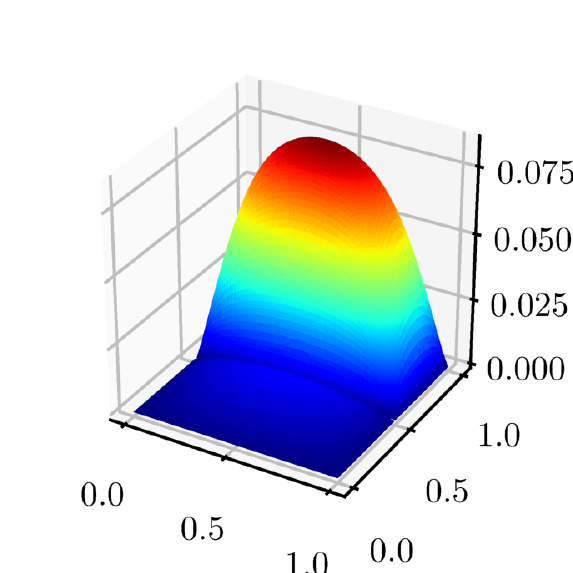}
  \caption{Solutions \(u\) on the two component domain corresponding to a constant right-hand side for \(s=0.25\) (\emph{top}) and for \(s=0.75\) (\emph{bottom}).}
  \label{fig:solutionsDisconnected}
\end{figure}

\begin{figure}
  \centering
  % \printPageSize
  \includegraphics{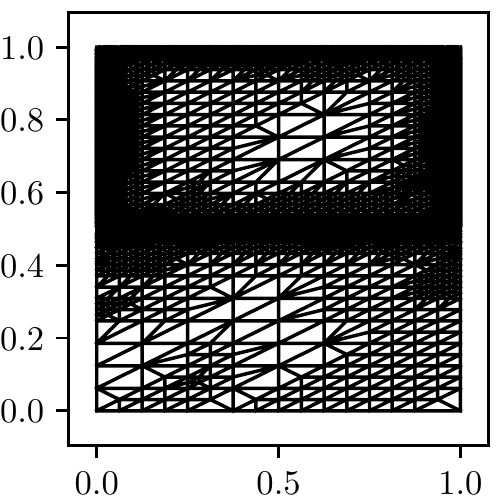}
  \includegraphics{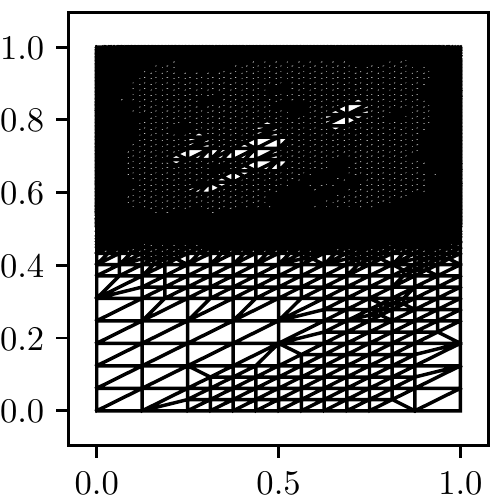}
  \caption{
    Adaptively refined meshes for the two component domain.
    \(s=0.25\) on the left, \(s=0.75\) on the right.}
  \label{fig:meshesDisconnected}
\end{figure}

\subsection{Exploiting Symmetry}
\label{sec:exploiting-symmetry}

The examples of the previous sections displayed symmetries and periodicity in the shape of the domain and the right-hand side and therefore in the solution.
In particular, the solution shown in \Cref{fig:solutionsPlateau} suggests that one should be able to exploit symmetry to halve the number of degrees of freedom required.

Let \(\mat{B}\) be an orthogonal transformation such that for some \(K\), \(\mat{B}^{K}=\operatorname{Id}\), \(\mat{B}\Omega=\Omega\), and let \(\omega\subset\tilde{\omega}\subset \mathbb{R}^{d}\) be representative domains such that \(\mathbb{R}^{d}=\bigcup_{i=0}^{K-1}\mat{B}^{i}\tilde{\omega}\) and \(\Omega=\bigcup_{i=0}^{K-1}\mat{B}^{i}\omega\). Moreover, assume that \(u\left(\mat{B}^{i}\vec{x}\right)=u\left(\vec{x}\right)\) and \(f\left(\mat{B}^{i}\vec{x}\right)=f\left(\vec{x}\right)\).
Examples of such transformations for \(\Omega=B(0,1)\) include the reflection \(\mat{B}=-\operatorname{Id}\) with \(\tilde{\omega}\) and \(\omega\) being the half-space and the half disc, or the rotation by angle \(\theta=\frac{2\pi}{K}\) and the domains \(\tilde{\omega}=\left\{(r\cos\phi, r\sin\phi)\in\mathbb{R}^{2}\mid r\geq 0, \phi\in[0,2\pi/K]\right\}\) and \(\omega=\tilde{\omega}\cap\left\{r\leq 1\right\}\).

The variational formulation of the fractional Poisson problem \eqref{eq:fracPoissonVariational} then reduces to
\begin{align*}
  a\left(u,v\right)
  &= \frac{C(d,s)}{2}\int_{\omega}\int_{\omega}\sum_{i,j=0}^{K}\frac{\left(u\left(\mat{B}^{i}\vec{x}\right)-u\left(\mat{B}^{j}\vec{y}\right)\right)\left(v\left(\mat{B}^{i}\vec{x}\right)-v\left(\mat{B}^{j}\vec{y}\right)\right)}{\abs{\mat{B}^{i}\vec{x}-\mat{B}^{j}\vec{y}}^{d+2s}} \\
  &\quad+ C(d,s)\int_{\omega}\int_{\tilde{\omega}\setminus \omega}\sum_{i,j=0}^{K}\frac{u\left(\mat{B}^{i}\vec{x}\right)v\left(\mat{B}^{i}\vec{x}\right)}{\abs{\mat{B}^{i}\vec{x}-\mat{B}^{j}\vec{y}}^{d+2s}}\\
  &= K\frac{C(d,s)}{2}\int_{\omega}\int_{\omega} \left(u\left(\vec{x}\right)-u\left(\vec{y}\right)\right)\left(v\left(\vec{x}\right)-v\left(\vec{y}\right)\right) k_{\text{eff}}\left(\vec{x},\vec{y}\right)  \\
  &\quad+ K C(d,s)\int_{\omega}\int_{\tilde{\omega}\setminus \omega}u\left(\vec{x}\right)v\left(\vec{x}\right)k_{\text{eff}}\left(\vec{x},\vec{y}\right)
    \intertext{and}
    \left\langle f, v\right\rangle
  &= \sum_{i=0}^{K-1} \int_{\omega}f\left(\mat{B}^{i}\vec{x}\right)v\left(\mat{B}^{i}\vec{x}\right) \\
  &= K \int_{\omega}fv
\end{align*}
with the effective kernel
\begin{align*}
  k_{\text{eff}}\left(\vec{x},\vec{y}\right)&=\sum_{i=0}^{K-1}\abs{\vec{x}-\mat{B}^{i}\vec{y}}^{-d-2s}.
\end{align*}
that accounts for the interaction with mirror images.
Effectively, the computational effort for the assembly and solution of systems involving the dense stiffness matrix is divided by a factor of \(K\), since only the sub-domain \(\omega\) needs to be discretized.
The complexity of the cluster method is expected to decrease from \(\mathcal{O}\left(n\left(\log n\right)^{2d}\right)\) to \(\mathcal{O}\left(n/K\left(\log n\right)^{2d}\right)\) since the assembly of the near-field contributions constitutes the bulk of the work.

To illustrate the approach, we solve the fractional Poisson problem with right-hand side \(f_{2,8}\) on the unit disc.
By leveraging the periodicity in the angular direction (and neglecting the mirror symmetry), we reduce to a problem on the circular sector \(\tilde{\omega}=\left\{(r\cos\phi, r\sin\phi)\in\mathbb{R}^{2}\mid r\geq 0, \phi\in[0,\pi/4]\right\}\) and \(\omega=\tilde{\omega}\cap\left\{r\leq 1\right\}\).
For simplicity, we employ a uniform mesh and assemble the dense system matrix.
The solution is shown in \Cref{fig:periodicSolution}.

\begin{figure}
  \centering
  \includegraphics{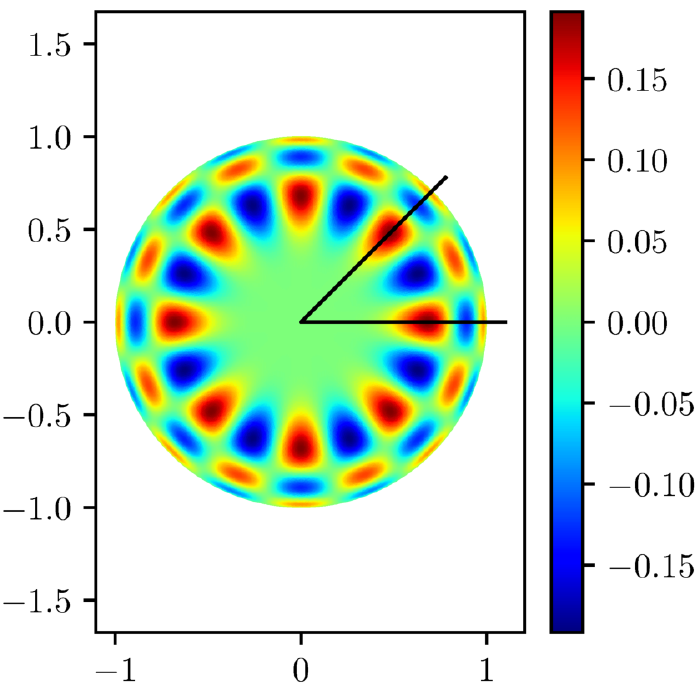}
  \caption{
    Periodic solution \(u=u_{2,8}^{2D}\).
    Only the circular sector on the right is discretized, with periodic images taken into account through the effective kernel \(k_{\text{eff}}\).
  }
  \label{fig:periodicSolution}
\end{figure}

\section{Conclusion}
\label{sec:conclusion}

In this work, we have extended the numerical approximation of the fractional Laplacian operator from the case of globally quasi-uniform meshes to locally refined meshes.
In order to overcome slow convergence due to the inherent singularity of the solution close to the boundary, we introduced a posteriori error indicators and adaptive mesh refinement.
The non-locality of the problem leads to dense matrices and solution complexity of \(\mathcal{O}\left(n^{3}\right)\) using direct solvers and at best \(\mathcal{O}\left(n^{2}\right)\) using an optimal iterative method.
We showed that by employing a cluster method, assembly, matrix-vector product and computation of the error estimators scale quasi-linear in the number of unknowns.
This means that the solution of the fractional Poisson problem \eqref{eq:fracPoisson} and the fractional Heat equation \eqref{eq:fractionalHeat} can be performed in the same complexity and with equal (or even better) rate of convergence as of the standard integer order equivalents.
Through several one- and two-dimensional examples, we illustrated that the predicted optimal rates of convergence are obtained.
We demonstrated how symmetries and periodicity in the problem can be leveraged to reduce the problem size even further.

\FloatBarrier

\appendix
\gdef\thesection{\Alph{section}} % corrected redefinition of "\thesection"
\makeatletter
\renewcommand\@seccntformat[1]{Appendix \csname the#1\endcsname.\hspace{0.5em}}
\makeatother

\section{Proof of \Cref{lem:strongForm}}
\label{sec:proof}

\strongForm*
\begin{proof}
  Let \(u\in V_{h}\) and \(\vec{x}\in K_{0}\), for some \(K_{0}\in \mathcal{P}_{h}\).
  Then
  \begin{align}
    \frac{\left(-\Delta\right)^{s}u\left(\vec{x}\right)}{C\left(d,s\right)}
    &= \operatorname{p.v.}\int_{\mathbb{R}^{d}} \frac{u\left(\vec{x}\right)-u\left(\vec{y}\right)}{\abs{\vec{x}-\vec{y}}^{d+2s}} \dd \vec{y} \nonumber\\
    &= \operatorname{p.v.}\int_{K_{0}} \frac{u\left(\vec{x}\right)-u\left(\vec{y}\right)}{\abs{\vec{x}-\vec{y}}^{d+2s}} \dd \vec{y}
      + u\left(\vec{x}\right) \int_{\mathbb{R}^{d}\setminus K_{0}} \frac{1}{\abs{\vec{x}-\vec{y}}^{d+2s}} \dd \vec{y}\nonumber\\
      &\quad-\sum_{K\neq K_{0}} \int_{K} \frac{u\left(\vec{y}\right)}{\abs{\vec{x}-\vec{y}}^{d+2s}}\dd \vec{y}.\label{eq:1}
  \end{align}

  We have the following helpful identities:
  \begin{align}
    \frac{\vec{x}-\vec{y}}{\abs{\vec{x}-\vec{y}}^{d+2s}}
    &=
      \begin{cases}
        \frac{1}{d+2s-2} \grad_{\vec{y}} \frac{1}{\abs{\vec{x}-\vec{y}}^{d+2s-2}}, & s\neq 1-d/2, \\
        \grad_{\vec{y}}\log \frac{1}{\abs{\vec{x}-\vec{y}}}, & s=1-d/2,
      \end{cases} \label{eq:2}\\
    \frac{1}{\abs{\vec{x}-\vec{y}}^{d+2s}}
    &= \frac{1}{2s} \grad_{\vec{y}}\cdot \frac{\vec{x}-\vec{y}}{\abs{\vec{x}-\vec{y}}^{d+2s}},  \label{eq:3}\\
    \frac{1}{\abs{\vec{x}-\vec{y}}^{d+2s}}
    &=
      \begin{cases}
        \frac{1}{2s(d+2s-2)}\Delta_{\vec{y}}\frac{1}{\abs{\vec{x}-\vec{y}}^{d+2s-2}}, & s\neq 1-d/2, \\
        \frac{1}{d-2} \Delta_{\vec{y}} \log \frac{1}{\abs{\vec{x}-\vec{y}}}, & s=1-d/2.
      \end{cases} \label{eq:4}
  \end{align}
  For \(\vec{x},\vec{y}\in K_{0}\),
  \begin{align*}
    u\left(\vec{x}\right)-u\left(\vec{y}\right)&= \restr{\grad u}{K_{0}}\cdot\left(\vec{x}-\vec{y}\right)
  \end{align*}
  since \(\restr{u}{K_{0}}\) is linear.
  Hence, the first term of \eqref{eq:1} can be rewritten using \eqref{eq:2} as
  \begin{align*}
    \operatorname{p.v.}\int_{K_{0}} \frac{u\left(\vec{x}\right)-u\left(\vec{y}\right)}{\abs{\vec{x}-\vec{y}}^{d+2s}} \dd \vec{y}
    &= \lim_{\varepsilon\rightarrow0} \restr{\grad u}{K_{0}} \cdot \int_{K_{0}\setminus B(\vec{x},\varepsilon)}  \frac{\vec{x}-\vec{y}}{\abs{\vec{x}-\vec{y}}^{d+2s}} \dd \vec{y}\\
    &= \lim_{\varepsilon\rightarrow0} \left[\restr{\grad u}{K_{0}} \cdot \frac{1}{d+2s-2}\int_{\partial K_{0}} \frac{\vec{n}_{\vec{y}}}{\abs{\vec{x}-\vec{y}}^{d+2s-2}} \dd \vec{y} \right. \\
    &\qquad+ \left.\restr{\grad u}{K_{0}} \cdot \frac{1}{d+2s-2} \underbrace{\int_{\partial B(\vec{x},\varepsilon)} \frac{\vec{n}_{\vec{y}}}{\abs{\vec{x}-\vec{y}}^{d+2s-2}} \dd \vec{y}}_{=0}
      \right] \\
    &= \frac{1}{d+2s-2}\int_{\partial K_{0}} \frac{\restr{\grad u}{K_{0}}\cdot\vec{n}_{\vec{y}}}{\abs{\vec{x}-\vec{y}}^{d+2s-2}} \dd \vec{y}.
  \end{align*}
  Here, \(\vec{n}_{\vec{y}}\) is the outward normal at \(\vec{y}\).
  When \(s=1-d/2\), we obtain instead that
  \begin{align*}
    \operatorname{p.v.}\int_{K_{0}} \frac{u\left(\vec{x}\right)-u\left(\vec{y}\right)}{\abs{\vec{x}-\vec{y}}^{2}} \dd \vec{y}
    &=\int_{\partial K_{0}} \restr{\grad u}{K_{0}}\cdot\vec{n}_{\vec{y}} \log \frac{1}{\abs{\vec{x-\vec{y}}}} \dd \vec{y}.
  \end{align*}

  The second term of \eqref{eq:1} can be rewritten using \eqref{eq:3} as
  \begin{align*}
    \int_{\mathbb{R}^{d}\setminus K_{0}} \frac{1}{\abs{\vec{x}-\vec{y}}^{d+2s}} \dd \vec{y}
    &= \frac{1}{2s} \int_{\mathbb{R}^{d}\setminus K_{0}} \grad_{\vec{y}}\frac{\vec{x}-\vec{y}}{\abs{\vec{x}-\vec{y}}^{d+2s}} \dd \vec{y} \\
    &= -\frac{1}{2s} \int_{\partial K_{0}} \frac{\vec{n}_{\vec{y}}\cdot\left(\vec{x}-\vec{y}\right)}{\abs{\vec{x}-\vec{y}}^{d+2s}} \dd \vec{y}.
  \end{align*}
  Finally, for every \(K\in\mathcal{P}_{h}\), \(K\neq K_{0}\), it follows from \eqref{eq:4}, Green's second identity and \eqref{eq:2} that
  \begin{align*}
    \int_{K}  \frac{u\left(\vec{y}\right)}{\abs{\vec{x}-\vec{y}}^{d+2s}} \dd \vec{y}
    &= \frac{1}{2s(d+2s-2)}\int_{K}  u\left(\vec{y}\right) \Delta_{\vec{y}}\frac{1}{\abs{\vec{x}-\vec{y}}^{d+2s-2}} \dd \vec{y}\\
    &= \frac{1}{2s(d+2s-2)}\int_{K} \underbrace{\left(\Delta_{\vec{y}} u\left(\vec{y}\right)\right)}_{=0} \frac{1}{\abs{\vec{x}-\vec{y}}^{d+2s-2}} \dd \vec{y}\\
    &\quad + \frac{1}{2s(d+2s-2)}\int_{\partial K}  \left\{u\left(\vec{y}\right) (d+2s-2) \frac{\vec{n}_{\vec{y}}\cdot\left(\vec{x}-\vec{y}\right)}{\abs{\vec{x}-\vec{y}}^{d+2s}} - \restr{\grad u}{K} \cdot \vec{n}_{\vec{y}} \frac{1}{\abs{\vec{x}-\vec{y}}^{d+2s-2}}\right\}\dd \vec{y} \\
    &= \frac{1}{2s}\int_{\partial K}  u\left(\vec{y}\right) \frac{\vec{n}_{\vec{y}}\cdot\left(\vec{x}-\vec{y}\right)}{\abs{\vec{x}-\vec{y}}^{d+2s}}\dd \vec{y}
      - \frac{1}{2s(d+2s-2)} \int_{\partial K}\restr{\grad u}{K} \cdot \vec{n}_{\vec{y}} \frac{1}{\abs{\vec{x}-\vec{y}}^{d+2s-2}}\dd \vec{y}.
  \end{align*}
  When \(s=1-d/2\), we obtain instead that
  \begin{align*}
    \int_{K}  \frac{u\left(\vec{y}\right)}{\abs{\vec{x}-\vec{y}}^{d+2s}} \dd \vec{y}
    &= \frac{1}{2s}\int_{\partial K}  u\left(\vec{y}\right) \frac{\vec{n}_{\vec{y}}\cdot\left(\vec{x}-\vec{y}\right)}{\abs{\vec{x}-\vec{y}}^{2}}\dd \vec{y}
      - \frac{1}{2s} \int_{\partial K}\restr{\grad u}{K} \cdot \vec{n}_{\vec{y}} \log\frac{1}{\abs{\vec{x}-\vec{y}}}\dd \vec{y}.
  \end{align*}
  In conclusion, we find that
  \begin{align*}
    \frac{\left(-\Delta\right)^{s}u\left(\vec{x}\right)}{C\left(d,s\right)}
    &= \frac{1}{d+2s-2}\int_{\partial K_{0}} \frac{\restr{\grad u}{K_{0}}\cdot\vec{n}_{\vec{y}}}{\abs{\vec{x}-\vec{y}}^{d+2s-2}} \dd \vec{y}
      -\frac{u\left(\vec{x}\right)}{2s} \int_{\partial K_{0}} \frac{\vec{n}_{\vec{y}}\cdot\left(\vec{x}-\vec{y}\right)}{\abs{\vec{x}-\vec{y}}^{d+2s}} \dd \vec{y} \\
    &\quad + \sum_{K\neq K_{0}} \left\{\frac{1}{2s(d+2s-2)} \int_{\partial K}\restr{\grad u}{K} \cdot \vec{n}_{\vec{y}} \frac{1}{\abs{\vec{x}-\vec{y}}^{d+2s-2}}\dd \vec{y}
      - \frac{1}{2s}\int_{\partial K} u\left(\vec{y}\right) \frac{\vec{n}_{\vec{y}}\cdot\left(\vec{x}-\vec{y}\right)}{\abs{\vec{x}-\vec{y}}^{d+2s}}\dd \vec{y}\right\}
  \end{align*}
  and when \(s=1-d/2\) then
  \begin{align*}
    \frac{\left(-\Delta\right)^{s}u\left(\vec{x}\right)}{C\left(d,s\right)}
    &= \int_{\partial K_{0}} \restr{\grad u}{K_{0}}\cdot\vec{n}_{\vec{y}} \log \frac{1}{\abs{\vec{x-\vec{y}}}} \dd \vec{y}
      -\frac{u\left(\vec{x}\right)}{2s} \int_{\partial K_{0}} \frac{\vec{n}_{\vec{y}}\cdot\left(\vec{x}-\vec{y}\right)}{\abs{\vec{x}-\vec{y}}^{d+2s}} \dd \vec{y} \\
    &\quad + \sum_{K\neq K_{0}} \left\{\frac{1}{2s} \int_{\partial K}\restr{\grad u}{K} \cdot \vec{n}_{\vec{y}} \log\frac{1}{\abs{\vec{x}-\vec{y}}}\dd \vec{y}
      -\frac{1}{2s}\int_{\partial K}  u\left(\vec{y}\right) \frac{\vec{n}_{\vec{y}}\cdot\left(\vec{x}-\vec{y}\right)}{\abs{\vec{x}-\vec{y}}^{2}}\dd \vec{y}  \right\}.
  \end{align*}
\end{proof}

\section{Determination of Quadrature Orders}
\label{sec:determ-quadr-orders}

\begin{theorem}[\cite{SauterSchwab2010_BoundaryElementMethods}, Theorems 5.3.23 and 5.3.24]\label{thm:localError}
  If \(K\) and \(\tilde{K}\) (\(K\) and \(e\)) are touching elements, then
  \begin{align*}
    \abs{E_{K\times\tilde{K}}^{i,j}}&\leq Ch_{K}^{2-2s}\rho_{1}^{-2k_{T}},\\
    \abs{E_{K\times e}^{i,j}} &\leq Ch_{K}^{2-2s}\rho_{3}^{-2k_{T,\partial}},
  \end{align*}
  where \(\rho_{1},\rho_{3}>1\) and \(k_{T}\), \(k_{T,\partial}\) are the quadrature orders in every dimension of \cref{eq:elementPair,eq:elementPairBoundary} (after regularisation).

  If \(K\) and \(\tilde{K}\) (\(K\) and \(e\)) are not touching, then
  \begin{flalign*}
    \abs{E_{K\times\tilde{K}}^{i,j}} &\leq C \left(h_{K}h_{\tilde{K}}\right)^{2} d_{K,\tilde{K}}^{-2-2s}\left\{\left(\frac{d_{K,\tilde{K}}}{h_{K}}\right)^{2}\tilde{\rho}_{2}\left(K,\tilde{K}\right)^{-2k_{NT}} + \left(\frac{d_{K,\tilde{K}}}{h_{\tilde{K}}}\right)^{2}\tilde{\rho}_{2}\left(\tilde{K},K\right)^{-2k_{NT}}\right\},\\
    \abs{E_{K\times e}^{i,j}} &\leq C \left(h_{K}h_{e}\right)^{2} d_{K,e}^{-2-2s}\left\{\left(\frac{d_{K,e}}{h_{K}}\right)^{2}\tilde{\rho}_{4}\left(K,e\right)^{-2k_{NT}} + \left(\frac{d_{K,e}}{h_{e}}\right)^{2}\tilde{\rho}_{4}\left(e,K\right)^{-2k_{NT}}\right\},
  \end{flalign*}
  where \(d_{K,\tilde{K}}:=dist(K,\tilde{K})\), \(d_{K,e}:=dist(K,e)\), \(\tilde{\rho}_{2}(K,\tilde{K}):=\rho_{2}\max\left\{\frac{d_{K,\tilde{K}}}{h_{K}},1\right\}\), \(\tilde{\rho}_{4}(K,e):=\rho_{4}\max\left\{\frac{d_{K,e}}{h_{K}},1\right\}\) and \(\rho_{2},\rho_{4}>1\), and \(k_{NT}\), \(k_{NT,\partial}\) are the quadrature order in every dimension of \cref{eq:elementPair,eq:elementPairBoundary}.
\end{theorem}

\begin{theorem}\label{thm:quadError}
  For \(d=2\), let \(\mathcal{I}_{K}\) index the degrees of freedom on \(K\in\mathcal{P}_{h}\), and define \(\mathcal{I}_{K\times\tilde{K}}:=\mathcal{I}_{K}\cup\mathcal{I}_{\tilde{K}}\).
  Let \(k_{T}\left(K,\tilde{K}\right)\) (respectively \(k_{T,\partial}\left(K,e\right)\)) be the quadrature order used for touching pairs \(K\times\tilde{K}\) (respectively \(K\times e\)), and let \(k_{NT}\left(K,\tilde{K}\right)\) (respectively \(k_{NT,\partial}\left(K,e\right)\)) be the quadrature order used for pairs that have empty intersection.
  Denote the resulting approximation to the bilinear form \(a\left(\cdot,\cdot\right)\) by \(a^{Q}\left(\cdot,\cdot\right)\).
  Then the  consistency error due to quadrature is bounded by
  \begin{align*}
    \abs{a(u,v)-a_{Q}(u,v)}&\leq C \left(E_{T} + E_{NT} + E_{T,\partial} + E_{NT,\partial}\right) \norm{u}_{L^{2}\left(\Omega\right)} \norm{v}_{L^{2}\left(\Omega\right)} \quad \forall u,v \in V_{h},
  \end{align*}
  where the errors are given by
  \begin{align*}
    E_{T}&= n \max_{K,\tilde{K}\in\mathcal{P}_{h}, \overline{K}\cap\overline{\tilde{K}}\not=\emptyset}h_{K}^{-2s}\rho_{1}^{-2k_{T}\left(K,\tilde{K}\right)},\\
    E_{NT} &= n \max_{K,\tilde{K}\in\mathcal{P}_{h}, \overline{K}\cap\overline{\tilde{K}}=\emptyset} \left[h_{\tilde{K}}^{2}\tilde{\rho}_{2}\left(K,\tilde{K}\right)^{-2k_{NT}\left(K,\tilde{K}\right)} + h_{K}^{2}\tilde{\rho}_{2}\left(\tilde{K},K\right)^{-2k_{NT}\left(K,\tilde{K}\right)}\right] d_{K,\tilde{K}}^{-2s} \max\left\{h_{K}^{-2},h_{\tilde{K}}^{-2}\right\}, \\
    E_{T,\partial} &= n^{1/2}\max_{K\in\mathcal{P}_{h},e\in\partial\mathcal{P}_{h},\overline{K}\cap\overline{e}\not=\emptyset}h_{K}^{-2s}\rho_{3}^{-2k_{T,\partial}\left(K,e\right)},\\
    E_{NT,\partial} &= n^{1/2}\max_{K\in\mathcal{P}_{h}, e\in\partial\mathcal{P}_{h}, \overline{K}\cap\overline{e}=\emptyset} \left[h_{e}^{2}\tilde{\rho}_{4}\left(K,e\right)^{-2k_{NT}\left(K,e\right)} + h_{K}^{2}\tilde{\rho}_{4}\left(e,K\right)^{-2k_{NT}\left(K,e\right)}\right] d_{K,e}^{-2s}h_{K}^{-2},
  \end{align*}
  and where
  \(d_{K,\tilde{K}}:=\inf_{\vec{x}\in K, \vec{y}\in\tilde{K}}\abs{\vec{x}-\vec{y}}\), \(d_{K,e}:=\inf_{\vec{x}\in K, \vec{y}\in e}\abs{\vec{x}-\vec{y}}\), and \(\rho_{1},\tilde{\rho}_{2},\rho_{3},\tilde{\rho}_{4}\) are as in \Cref{thm:localError}.
\end{theorem}

Based on the error estimates of \Cref{thm:quadError}, the optimal quadrature orders in order to allow convergence of up to order \(n^{-\alpha}\) should be chosen as
\begin{align*}
  k_{T}\left(K,\tilde{K}\right)
  &\geq \frac{\left(\alpha/2+1/2\right)\log n +s\abs{\log h_{K}}}{\log(\rho_{1})}-C,\\
  k_{NT}\left(K,\tilde{K}\right)
  &\geq \max\left\{\frac{\left(\alpha/2+1/2\right)\log n +(s-1)\abs{\log h_{\tilde{K}}} + \max\left\{\abs{\log h_{K}},\abs{\log h_{\tilde{K}}}\right\}- s\log\frac{d_{K,\tilde{K}}}{h_{\tilde{K}}} - C}{\log_{+}\frac{d_{K,\tilde{K}}}{h_{K}} + \log(\rho_{2})}, \right.\\
  &\left.\qquad\frac{\left(\alpha/2+1/2\right)\log n +(s-1)\abs{\log h_{K}} + \max\left\{\abs{\log h_{K}},\abs{\log h_{\tilde{K}}}\right\}- s\log\frac{d_{K,\tilde{K}}}{h_{K}} - C}{\log_{+}\frac{d_{K,\tilde{K}}}{h_{\tilde{K}}} + \log(\rho_{2})}
    \right\}\\
  k_{T,\partial}\left(K,e\right)
  &\geq \frac{\left(\alpha/2+1/4\right)\log n + s\abs{\log h_{K}}}{\log(\rho_{3})}-C,\\
  k_{NT,\partial}\left(K,e\right)
  &\geq \max\left\{\frac{\left(\alpha/2+1/4\right)\log n +(s-1)\abs{\log h_{e}} + \max\left\{\abs{\log h_{K}},\abs{\log h_{e}}\right\}- s\log\frac{d_{K,e}}{h_{e}} - C}{\log_{+}\frac{d_{K,e}}{h_{K}} + \log(\rho_{4})}, \right.\\
  &\left.\qquad\frac{\left(\alpha/2+1/4\right)\log n +(s-1)\abs{\log h_{K}} + \max\left\{\abs{\log h_{K}},\abs{\log h_{e}}\right\}- s\log\frac{d_{K,e}}{h_{K}} - C}{\log_{+}\frac{d_{K,e}}{h_{e}} + \log(\rho_{4})}
    \right\}.
\end{align*}

\begin{proof}[Proof of \Cref{thm:quadError}]
  Let the quadrature rules for the pairs \(K\times\tilde{K}\) and \(K\times e\) be denoted by \(a^{K\times\tilde{K}}_{Q}\left(\cdot,\cdot\right)\) and \(a^{K\times e}_{Q}\left(\cdot,\cdot\right)\).
  Set
  \begin{align*}
    E_{K\times\tilde{K}}^{i,j} &= a^{K\times\tilde{K}}\left(\phi_{i},\phi_{j}\right)-a_{Q}^{K\times\tilde{K}}\left(\phi_{i},\phi_{j}\right),\\
    E_{K\times e}^{i,j} &= a^{K\times e}\left(\phi_{i},\phi_{j}\right)-a_{Q}^{K\times e}\left(\phi_{i},\phi_{j}\right).
  \end{align*}
  For \(u,v\in V_{h}\), we set
  \begin{align*}
    E_{K\times\tilde{K}}(u,v)&= \sum_{i\in\mathcal{I}_{K\times\tilde{K}}} \sum_{j\in\mathcal{I}_{K\times\tilde{K}}} u_{i}v_{j}E_{K\times\tilde{K}}^{i,j}, \\
    E_{K\times e}(u,v)&= \sum_{i\in\mathcal{I}_{K}} \sum_{j\in\mathcal{I}_{K}} u_{i}v_{j}E_{K\times e}^{i,j}
  \end{align*}
  so that
  \begin{align*}
    \abs{E_{K\times\tilde{K}}(u,v)}
    &\leq \left(\max_{i,j}\abs{E_{K\times\tilde{K}}^{i,j}}\right) \sum_{i\in\mathcal{I}_{K\times\tilde{K}}} \abs{u_{i}} \sum_{j\in\mathcal{I}_{K\times\tilde{K}}} \abs{v_{j}} \\
    &\leq \left(\max_{i,j}\abs{E_{K\times\tilde{K}}^{i,j}}\right) \abs{\mathcal{I}_{K\times\tilde{K}}} \sqrt{\sum_{i\in\mathcal{I}_{K\times\tilde{K}}} \abs{u_{i}}^{2}} \sqrt{\sum_{j\in\mathcal{I}_{K\times\tilde{K}}} \abs{v_{j}}^{2}}, \\
    \abs{E_{K\times e}(u,v)}
    &\leq \left(\max_{i,j}\abs{E_{K, e}^{i,j}}\right) \sum_{i\in\mathcal{I}_{K}} \abs{u_{i}} \sum_{j\in\mathcal{I}_{K}} \abs{v_{j}} \\
    &\leq \left(\max_{i,j}\abs{E_{K, e}^{i,j}}\right) \abs{\mathcal{I}_{K}} \sqrt{\sum_{i\in\mathcal{I}_{K}} \abs{u_{i}}^{2}} \sqrt{\sum_{j\in\mathcal{I}_{K}} \abs{v_{j}}^{2}}
  \end{align*}
  Since
  \begin{align*}
    \sum_{i\in\mathcal{I}_{K\times\tilde{K}}} \abs{u_{i}}^{2}
    &\leq C\left[h_{K}^{-d}\int_{K}u^{2} + h_{\tilde{K}}^{-d}\int_{\tilde{K}}u^{2}\right], \\
    \sum_{i\in\mathcal{I}_{K}} \abs{u_{i}}^{2}
    &\leq C h_{K}^{-d}\int_{K}u^{2},
  \end{align*}
  we find
  \begin{align*}
    \abs{a(u,v)-a_{Q}(u,v)}
    &\leq \sum_{K}\sum_{\tilde{K}} \abs{E_{K\times\tilde{K}}(u,v)} + \sum_{K}\sum_{e} \abs{E_{K\times e}(u,v)} \\
    &\leq C\sum_{K}\sum_{\tilde{K}} \left(\max_{i,j}\abs{E_{K\times\tilde{K}}^{i,j}}\right) \max\left\{h_{K}^{-d},h_{\tilde{K}}^{-d}\right\} \left[\norm{u}_{L^{2}(K)}^{2} + \norm{u}_{L^{2}(\tilde{K})}^{2}\right]^{1/2} \\
    &\qquad\qquad \left[ \norm{v}_{L^{2}(K)}^{2} + \norm{v}_{L^{2}(\tilde{K})}^{2}\right]^{1/2} \\
    &\quad + C\sum_{K}\sum_{e} \left(\max_{i,j}\abs{E_{K\times e}^{i,j}}\right)  h_{K}^{-d}\norm{u}_{L^{2}(K)} \norm{v}_{L^{2}(K)} \\
    &\leq C \left(\max_{K,\tilde{K}}\max_{i,j}\abs{E_{K\times\tilde{K}}^{i,j}} \max\left\{h_{K}^{-d},h_{\tilde{K}}^{-d}\right\}\right) \sum_{K}\sum_{\tilde{K}} \norm{u}_{L^{2}(K\cup\tilde{K})} \norm{v}_{L^{2}(K\cup\tilde{K})} \\
    &\quad + C \left(\max_{K,e}\max_{i,j}\abs{E_{K\times e}^{i,j}}h_{K}^{-d}\right) \sum_{K}\sum_{e} \norm{u}_{L^{2}(K)} \norm{v}_{L^{2}(K)}.
  \end{align*}
  Because
  \begin{align*}
    \sum_{K}\sum_{\tilde{K}} \norm{u}_{L^{2}(K\cup\tilde{K})} \norm{v}_{L^{2}(K\cup\tilde{K})}
    &\leq \sqrt{\sum_{K}\sum_{\tilde{K}} \norm{u}_{L^{2}(K\cup\tilde{K})}^{2}} \sqrt{\sum_{K}\sum_{\tilde{K}} \norm{v}_{L^{2}(K\cup\tilde{K})}^{2}} \\
    &\leq 2\abs{\mathcal{P}_{h}} \norm{u}_{L^{2}(\Omega)} \norm{v}_{L^{2}(\Omega)}\\
    &\leq Cn \norm{u}_{L^{2}(\Omega)} \norm{v}_{L^{2}(\Omega)}
      \intertext{and}
      \sum_{K}\sum_{e} \norm{u}_{L^{2}(K)} \norm{v}_{L^{2}(K)}
    &\leq \abs{\partial\mathcal{P}_{h}} \norm{u}_{L^{2}(\Omega)} \norm{v}_{L^{2}(\Omega)}\\
    &\leq Cn^{(d-1)/d} \norm{u}_{L^{2}(\Omega)} \norm{v}_{L^{2}(\Omega)},
  \end{align*}
  we obtain
  \begin{align*}
    \abs{a(u,v)-a_{Q}(u,v)}
    &\leq C \left[n\left(\max_{K,\tilde{K}}\max_{i,j}\abs{E_{K\times\tilde{K}}^{i,j}} \max\left\{h_{K}^{-d},h_{\tilde{K}}^{-d}\right\}\right) \right.\\
    &\qquad\left.+ n^{(d-1)/d}\left(\max_{K,e}\max_{i,j}\abs{E_{K\times e}^{i,j}} h_{K}^{-d}\right)\right] \norm{u}_{L^{2}(\Omega)} \norm{v}_{L^{2}(\Omega)}.
  \end{align*}
  For \(d=2\), using \Cref{thm:localError} stated below permits to conclude.
\end{proof}

% \printbibliography

\bibliographystyle{elsarticle-harv}
\bibliography{bibtex/papers}

\end{document}